\documentclass{article}
\usepackage{geometry}
\usepackage{graphicx} % Required for inserting images
\usepackage{amssymb}
\usepackage{amsthm}
\usepackage{lineno}
\usepackage{hyperref}
\usepackage{amsmath}
\usepackage{epsfig,epstopdf}
\usepackage{amsthm,amscd,mathrsfs}
\usepackage{latexsym}
\usepackage{graphicx}
\usepackage{extarrows}
\usepackage{array}
\usepackage[figuresright]{rotating}
\modulolinenumbers[5]
\newtheorem{theorem}{Theorem}[section]
\newtheorem{lemma}[theorem]{Lemma}

\newtheorem{example}[theorem]{Example}
\newtheorem{remark}[theorem]{Remark}

\usepackage{multirow}
\usepackage{diagbox}
\usepackage{color}
\usepackage{subfigure}
\usepackage{bm}
\allowdisplaybreaks
\usepackage[section]{placeins}
\renewcommand{\d}{{\rm d}}
\newcommand{\ud}{\mathrm{d}}
\title{Convergence and superconvergence analysis of discontinuous Galerkin methods for index-2 integral-algebraic equations}
\author{Hecong Gao and Hui Liang}
\date{October 2024}

\begin{document}

\maketitle
\begin{abstract}
		The integral-algebraic equation (IAE) is a mixed system of first-kind and second-kind Volterra integral equations (VIEs). This paper mainly focuses on the discontinuous Galerkin (DG) method to solve index-2 IAEs. First, the convergence theory of perturbed DG methods for first-kind VIEs is established, and then used to derive the optimal convergence properties of DG methods for index-2 IAEs. It is shown that an $(m-1)$-th degree DG approximation exhibits global convergence of order~$m$ when~$m$ is odd, and of order~$m-1$ when~$m$ is even, for the first component~$x_1$ of the exact solution, corresponding to the second-kind VIE, whereas the convergence order is reduced by two for the second component~$x_2$ of the exact solution, corresponding to the first-kind VIE. Each component also exhibits local superconvergence of one order higher when~$m$ is even. When~$m$ is odd, superconvergence occurs only if $x_1$ satisfies $x_1^{(m)}(0)=0$. Moreover, with this condition, we can extend the local superconvergence result for~$x_2$ to global superconvergence when~$m$ is odd. Note that in the DG method for an index-1 IAE, generally, the global superconvergence of the exact solution component corresponding to the second-kind VIE can only be obtained by iteration. However, we can get superconvergence for all components of the exact solution of the index-2 IAE directly. Some numerical experiments are given to illustrate the obtained theoretical results.
	\end{abstract}
\section{Introduction}
This paper mainly focuses on the following general linear index-2 integral-algebraic equation (IAE):
\begin{align}\label{IAE}
	\left\{\begin{aligned}
		x_1(t)+\int_{0}^{t}\left[K_{11}(t,s)x_1(s)+K_{12}(t,s)x_2(s)\right]\,\mathrm{d}s&=f_1(t),\\
		\int_{0}^{t}K_{21}(t,s)x_1(s)\,\mathrm{d}s&=f_2(t),
	\end{aligned}\right.
\end{align}
for $t\in I:=[0,T]$,
where $K_{ij}$~and $f_j$ are given real-valued functions satisfying
\[
f_2(0)=0\quad\text{and}\quad\text{$|K_{21}(t,t)K_{12}(t,t)|\geq k_0 >0$ for all $t\in I$.}
\]
The IAE \eqref{IAE} is a mixed system of first-kind and second-kind Volterra integral equations (VIEs), initially introduced by Gear~\cite{Gear1990}, who formulated the definition of differential index. IAEs have broad applications as mathematical models in various physical and biological contexts, garnering significant interest among researchers; see the monographs of Brunner~\cite{Brunner, Book2017} and references therein. Liang and Brunner~\cite{tract, 2016siam} introduced the tractability index for the general system of linear IAEs based on the $\nu$-smoothing property of a Volterra integral operator, and decoupled the IAE system into the inherent system of regular second-kind VIEs and a system of first-kind VIEs. In particular, the general linear index-2 IAE can be decoupled into the semi-explicit form~\eqref{IAE}.  It is known~\cite{Gear1990, tract} that both the differential index and the tractability index of the IAE~\eqref{IAE} equal~$2$.

Due to their rich applications, many works have focused on the numerical analysis of IAEs \cite{Brunner, tract, 2016siam, Hadizadeh2010, Hessenberg1, index3, gao2022}. In particular, for the index-2 IAE~\eqref{IAE}, the Jacobi spectral method was investigated by Hadizadeh~\cite{Hadizadeh2010}.  In other work, we have studied discontinuous piecewise polynomial collocation methods for higher-index IAEs of Hessenberg type~\cite{gao2022}, which also include the semi-explicit index-2 IAE~ \eqref{IAE}. Liang and Brunner~\cite{2016siam} observed that not all collocation methods are feasible for index-2 IAEs or higher-index IAEs, and established the convergence of collocation methods for the semi-explicit index-2 IAE~\eqref{IAE}. In addition, the existence, uniqueness and regularity of the exact solution of the IAE~\eqref{IAE} are given as follows.
%--------------------------------		
\begin{theorem}[{\cite[Theorem 4.3]{2016siam}}]\label{thr1}
	Assume that $f_i \in C^{d+i}(I)$ ($i=1$, $2$) with~$f_2(0)=0$, $K_{11}\in C^{d+1}(D)$, $K_{12}$, $K_{21}\in C^{d+2}(D)$ with~$D:=\{\,(t,s):0\leq s \leq t \leq T\,\}$. Then, the system~\eqref{IAE} has a unique solution $x=(x_1,x_2)^{T}\in C^d(I)$, and there exist functions $\kappa_{1i}, \kappa_{2j}\in C^d(I) (i=1,2;j=1,\dots,5)$ and $Q_{11},Q_{2p}\in C^d(D)(p=1,2)$, such that the solution can be represented in the form
	\begin{align*}
		x_1(t)&=\kappa_{11}(t)f_2(t)+\kappa_{12}(t)f'_2(t)+\int_{0}^{t}Q_{11}(t,s)f_2(s)\,\mathrm{d}s, \\
		x_2(t)&=\kappa_{21}(t)f_1(t)+\kappa_{22}(t)f_1'(t)+\kappa_{23}(t)f_2(t)+\kappa_{24}(t)f_2'(t)+\kappa_{25}(t)f_2''(t)+\sum_{p=1}^2\int_{0}^{t}Q_{2p}(t,s)f_p(s)\,\mathrm{d}s.
	\end{align*}
\end{theorem}
%--------------------------------	

For VIEs, piecewise polynomial collocation is one of the most popular numerical methods~\cite{Brunner}, with discontinuous piecewise polynomials forming a natural numerical approximation space. Zhang et al.~\cite{DG1998} studied the discontinuous Galerkin (DG) method for second-kind VIEs, showing that the $(m-1)$-th degree DG approximation possesses global convergence of order~$m$. Moreover, the iterated DG solution is globally superconvergent of order~$m+1$, and achieves a local superconvergence order~$2m$ at the mesh points. Liang~\cite{DG2021} gave a convergence analysis of the DG method for second-kind VIEs without using operator theory. For first-kind VIEs, Brunner et al.~\cite{DG2009} showed that the $(m-1)$-th degree DG approximation exhibits global convergence of order~$m-1$ when~$m$ is even, and of order~$m$ when~$m$ is odd. There is local superconvergence of one order higher but, in the odd order case, only if the exact solution~$u$ of the equation satisfies $u^{(m)}(0)=0$.

The IAE is a mixed system of first-kind and second-kind VIEs, so a natural
question arises: What are the optimal convergence orders and the local
superconvergence orders of the components $x_1$~and $x_2$ of the DG solution of
the coupled IAE system? In our study of the DG method for the index-1
IAE~\cite{gao2023}, the optimal convergence orders were obtained for both the
components corresponding to the first-kind and second-kind VIEs. However, in
order to improve the numerical accuracy, the iterated DG method is introduced,
and the global superconvergence of the iterated DG solution for the component
corresponding to the second-kind VIE is derived. The index-2 IAE~\eqref{IAE} is
also a coupled system of first-kind and second-kind VIEs, but it can be
considered as the system of first-kind VIEs in the perturbed version, which is
different from the index-1 IAE. However, the iterated DG solution can not be
defined for first-kind VIEs. Therefore, the optimal convergence order and the
local superconvergence order of the DG solution for the index-2 IAE \eqref{IAE}
are still open, providing the motivation for this paper.  We aim to fill this
gap and give a complete convergence analysis of DG methods for the coupled
system of the index-2 IAE~\eqref{IAE}.  We will show that due to the special
structure of the index-2
IAE~\eqref{IAE}, fortunately, the superconvergence result can be obtained directly. Unlike the index-1 IAE, for which the superconvergence is obtained by virtue of the iterated DG method~\cite{gao2023}.

This paper is organized as follows. In Section~\ref{sec:2}, the DG scheme for the index-2 IAE~\eqref{IAE} is introduced. Section~\ref{sec:3} investigates the convergence analysis of perturbed DG methods for first-kind VIEs.  These results are then used in Section~\ref{sec:4} to establish the convergence of DG methods for the index-2 IAE, as well as superconvergence results. Numerical examples are given in Section~\ref{sec:5} to illustrate the theoretical results. Throughout the whole paper, $C$ denotes a generic constant, which may depend on the regularity of the solution, but does not depend on the stepsize~$h$.

	\section{DG approximation}\label{sec:2}
%----------------------------------------------
Let $N\geq 2$ be a given positive integer.  Define a uniform step
size~$h:=T/N$ and points~$t_n:=nh$ (so that $t_N=T$), and denote the resulting
mesh on~$I$ by
\[
I_h := \{\,t_n:  n=0, 1,  \ldots, N\,\}.
\]
We seek an approximate solution~$(x_{1,h}, x_{2,h})^T$ for the
IAE~\eqref{IAE}, with the components $x_{1,h}$~and $x_{2,h}$ in the space of
discontinuous piecewise polynomials
\[
S_{m-1}^{(-1)}(I_h):=\{\,v:
\text{$v|_{\sigma_n}\in\mathcal{P}_{m-1}(\sigma_n)$
	for $0\leq n\leq N-1)$}\,\}.
\]
Here, $\mathcal{P}_{m-1}(\sigma_n)$ denotes the space of all real polynomials
of degree not exceeding $m-1$ on the subinterval $\sigma_n:=(t_n,t_{n+1}]$.
Then, applying Galerkin's method, we require that
for all $\phi\in S_{m-1}^{(-1)}(I_h)$ and for $0\le n\le N-1$,
\begin{align}
	\int_{t_n}^{t_{n+1}}\left[x_{1,h}(s)+\int_{0}^{s}\left(
	K_{11}(s,\tau)x_{1,h}(\tau)+K_{12}(s,\tau)x_{2,h}(\tau)\right)
	\,\ud\tau\right]\phi(s)\,\ud s
	&=\int_{t_n}^{t_{n+1}}f_1(s)\phi(s)\ud s, \label{DGeq1} \\
	\int_{t_n}^{t_{n+1}}\int_{0}^{s}\left(
	K_{21}(s,\tau)x_{1,h}(\tau)\,\ud\tau\right)\phi(s)\,\ud s
	&=\int_{t_n}^{t_{n+1}}f_2(s)\phi(s)\,\ud s, \label{DGeq2}
\end{align}
i.e.,
\begin{align*}
	&h\int_{0}^{1}x_{1,h}(t_n+sh)\phi(t_n+sh)\,\ud s\\
	&{}+h^2\int_{0}^{1}\int_{0}^{s}\left[
	K_{11}(t_n+sh,t_n+\tau h)x_{1,h}(t_n+\tau h)
	+K_{12}(t_n+sh,t_n+\tau h)x_{2,h}(t_n+\tau h)
	\right]\,\ud\tau\,\phi(t_n+sh)\,\ud s\\
	&{}+h^2\sum_{l=0}^{n-1}\int_{0}^{1}\int_{0}^{1}\left[
	K_{11}(t_n+sh,t_l+\tau h)x_{1,h}(t_l+\tau h)
	+K_{12}(t_n+sh,t_l+\tau h)x_{2,h}(t_l+\tau h)
	\right]\,\ud\tau\,\phi(t_n+sh)\,\ud s\\
	&=h\int_{0}^{1}f_1(t_n+sh)\phi(t_n+sh)\,\ud s,
\end{align*}
and
\begin{align*}
	&h^2\int_{0}^{1}\left(\int_{0}^{s}
	K_{21}(t_n+sh,t_n+\tau h)x_{1,h}(t_n+\tau h)\,\ud\tau\right)
	\phi(t_n+sh)\,\ud s\\
	&\qquad{}+h^2\sum_{l=0}^{n-1}\int_{0}^{1}\left(\int_{0}^{1}
	K_{21}(t_n+sh,t_l+\tau h)x_{1,h}(t_l+\tau h)\,\ud\tau\right)
	\phi(t_n+sh)\,\ud s\\
	&=h\int_{0}^{1}f_2(t_n+sh)\phi(t_n+sh)\,\ud s.
\end{align*}
Since $x_{1,h}$, $x_{2,h}\in S_{m-1}^{(-1)}(I_h)$, the local representation of
the DG solution on the subinterval~$\sigma_n$ can be written as
\begin{equation}\label{local-DG}
	x_{p,h}(t_n+sh)=\sum_{j=0}^{m-1}P_j(s)U_{p,j}^n,
	\text{for $s\in(0,1]$ and $p=1$, $2$,}
\end{equation}
where $P_j(s)$ ($j=0$, \dots, $m-1$) denote the shifted Legendre polynomials of
degree~$j$ on~$[0,1]$, and where the $U_{p,j}^n$ ($p=1$, $2$) are
unknowns to be determined. For $i$, $j=0$, \dots, $m-1$ and $p$, $q=1$, $2$, set
\begin{gather*}
	a_{ij}:=\int_{0}^{1}P_j(s)P_i(s)\,\ud s, \quad
	\beta_{pq}^{n}(i,j):=\int_{0}^{1}\left(\int_{0}^{s}
	K_{pq}(t_n+sh,t_n+\tau h)P_j(\tau)\,\ud\tau\right)\,P_i(s)\,\ud s,\\
	\beta_{pq}^{(n,l)}(i,j):=\int_{0}^{1}\left(\int_{0}^{1}
	K_{pq}(t_n+sh,t_l+\tau h)P_j(\tau)\,\ud\tau\right)P_i(s)\,\ud s.
\end{gather*}
The Galerkin equations can then be written as
\[
\sum_{j=0}^{m-1}a_{ij}U_{1,j}^n+h\sum_{j=0}^{m-1}\left[
\sum_{p=1}^2\beta_{1p}^n(i,j)U_{p,j}^n\right]
+h\sum_{l=0}^{n-1}\sum_{j=0}^{m-1}\left[
\sum_{p=1}^2\beta_{1p}^{(n,l)}(i,j)U_{p,j}^l\right]
=\int_{0}^{1}f_1(t_n+sh)P_i(s)\,\ud s,
\]
and
\[
h\sum_{j=0}^{m-1}\beta_{21}^n(i,j)U_{1,j}^n
+h\sum_{l=0}^{n-1}\sum_{j=0}^{m-1}\beta_{21}^{(n,l)}(i,j)U_{1,j}^l
=\int_{0}^{1}f_2(t_n+sh)P_i(s)\,\ud s.
\]
For $p$, $q=1$, $2$, $0\leq l\leq n-1$ and $0\leq n\leq N-1$, define the
$m$-dimensional vectors
\[
\bm{F}_n^{p}:=\left(\int_{0}^{1}f_p(t_n+sh)P_0(s)\,\ud s, \ldots,
\int_{0}^{1}f_p(t_n+sh)P_{m-1}(s)\,\ud s\right)^T
\quad\text{and}\quad
\bm{U}_n^{p}:=\left(U_{p,0}^n, \ldots, U_{p,m-1}^n\right)^T,
\]
and the $m\times m$~matrices
\[
\bm{A}:=(a_{ij}),\quad
\bm{B}_{pq}^{(n,l)}:=(\beta_{pq}^{(n,l)}(i,j)),\quad
\bm{B}_{pq}^n:=(\beta_{pq}^{n}(i,j))
\quad\text{for $i$, $j=0$, \ldots, $m-1$.}
\]
The Galerkin equations can then be written in compact form
\begin{equation}\label{algebraic-eq}
	\begin{pmatrix}
		\bm{A}+h\bm{B}_{11}^n & h\bm{B}_{12}^n\\
		h\bm{B}_{21}^n & \bm{0}
	\end{pmatrix}
	\begin{pmatrix}
		\bm{U}^1_n\\
		\bm{U}^2_n
	\end{pmatrix}
	+h\sum_{l=0}^{n-1}\begin{pmatrix}
		\bm{B}_{11}^{(n,l)} & \bm{B}_{12}^{(n,l)} \\
		\bm{B}_{21}^{(n,l)} & \bm{0}
	\end{pmatrix}
	\begin{pmatrix}
		\bm{U}^1_l\\
		\bm{U}^2_l
	\end{pmatrix}
	=\begin{pmatrix}
		\bm{F}^1_n\\
		\bm{F}^2_n
	\end{pmatrix}.
\end{equation}

We now state some results from our earlier paper~\cite{gao2023} on the
elements $\beta_{pq}^{n}(i,j)$~and $\beta_{pq}^{(n,l)}(i,j)$ defined above.

%--------------------------------------
\begin{lemma}[{\cite[Lemma 2.1]{gao2023}}] \label{lemma1}
	If $K_{pq}\in C^{d+1}(D)$ for $p$, $q=1$, $2$ with~$d\geq 0$, then for all
	$i$, $j\geq 0$ and $0\leq n\leq N-1$,
	\[
	\beta_{pq}^{n}(i,j)=\begin{cases}
		\frac{1}{2}K_{pq}(t_n,t_n)+O(h),&\text{if $(i,j)=(0,0)$,}\\
		-\alpha_jK_{pq}(t_n,t_n)+O(h),&\text{if $i=j-1$ and $j\geq1$,}\\
		\alpha_iK_{pq}(t_n,t_n)+O(h),&\text{if $j=i-1$ and $i\geq 1$,}\\
		O(h),&\text{if $i=j>0$,} \\
		O(h^{|i-j|-1})+O(h^{d+1}),&\text{otherwise},
	\end{cases}
	\]
	where $\alpha_j=1/(8j^2-2)$ for $j\geq1$. Further, if $d\geq 2$, then for
	$l=0$, \dots, $n-1$,
	\[
	\beta_{pq}^{{(n,l)}}(i,j)=\begin{cases}
		K_{pq}(t_n,t_l)+K'_{pq}(t_n,t_l)h/2+d_{0,0}^{(n,l)}h^2/12+O(h^3),
		&\text{if $(i,j)=(0,0)$,}\\
		\partial_{j+1}K_{pq}(t_n,t_l)h/6+d_{i,j}^{(n,l)}h^2/12+O(h^3),
		&\text{if $(i,j)\in\{(0,1), (1,0)\}$,}\\
		d_{i,j}^{(n,l)}h^2/12+O(h^3),
		&\text{if $(i,j)\in\{(0,2), (2,0), (1,1)\}$,}\\
		O(h^{i+j})+O(h^{d+1}),&\text{otherwise},
	\end{cases}
	\]
	where $\partial_{r}^{k}K_{pq}(t_n,t_l)$ denotes the value at point $(t_n,t_l)$
	of $k$-th partial derivative of the function $K_{pq}(\cdot,\cdot)$ with respect
	to its $r$-th variable, and $K_{pq}'(t_n,t_l)$ denotes
	$(\partial_{1}+\partial_{2})K_{pq}(t_n,t_l)$. In addition,
	\begin{gather*}
		d_{0,0}^{(n,l)}:=2\partial_{1}^2K_{pq}(t_n,t_l)
		+3\partial_1\partial_2K_{pq}(t_n,t_l)+2\partial_{2}^2K_{pq}(t_n,t_l), \\
		d_{0,1}^{(n,l)}:=\partial_{2}^2K_{pq}(t_n,t_l)
		+\partial_{1}\partial_{2}K_{pq}(t_n,t_l), \quad
		d_{1,0}^{(n,l)}:=\partial_{1}^2K_{pq}(t_n,t_l)
		+\partial_{1}\partial_{2}K_{pq}(t_n,t_l), \\
		d_{0,2}^{(n,l)}:=\tfrac{1}{5}\partial_{2}^2K_{pq}(t_n,t_l), \quad
		d_{2,0}^{(n,l)}:=\tfrac{1}{5}\partial_{1}^2K_{pq}(t_n,t_l), \quad
		d_{1,1}^{(n,l)}:=\tfrac{1}{3}\partial_{1}\partial_{2}K_{pq}(t_n,t_l).
	\end{gather*}
\end{lemma}
%--------------------------------------	

We will use the $m\times m$ matrices
\begin{gather*}
	\bm{M}:=\begin{pmatrix}
		\frac{1}{2} & -\alpha_1 \\
		\alpha_1 & 0 & -\alpha_2 \\
		& \ddots & \ddots& \ddots \\
		& & \alpha_{m-2} & 0 & -\alpha_{m-1} \\
		& & & \alpha_{m-1} & 0
	\end{pmatrix}, \quad
	\bm{N}:=\begin{pmatrix}
		1  \\
		& 0 \\
		& & \ddots \\
		& & & 0
	\end{pmatrix},\\
	\bm{D}_{pq}^{(n,l)}:=\begin{pmatrix}
		3K_{pq}'(t_n,t_l)&\partial_2K_{pq}(t_n,t_l)&\cdots&0\\
		\partial_1K_{pq}(t_n,t_l)&0                        &\cdots&0 \\
		\vdots&                  \vdots &\ddots&\vdots\\
		0&                        0&\cdots&0
	\end{pmatrix}, \quad
	\tilde{\bm{D}}_{pq}^{(n,l)}:=\begin{pmatrix}
		d_{0,0}^{(n,l)}& d_{0,1}^{(n,l)}&d_{0,2}^{(n,l)}&\cdots&0\\
		d_{1,0}^{(n,l)}& d_{1,1}^{(n,l)}&              0&\cdots&0\\
		d_{2,0}^{(n,l)}&               0&              0&\cdots&0\\
		\vdots&          \vdots&         \vdots&\ddots&\vdots\\
		0&               0&              0&\cdots&0
	\end{pmatrix},
\end{gather*}
for $p$, $q=1$, $2$, $0\leq l\leq n-1$ and $0\leq n\leq N-1$.
%--------------------------------------
\begin{lemma}[{\cite[Lemma 2.2]{DG2009}}]\label{l:M}
	The matrix $\bm{M}$ is nonsingular.
\end{lemma}
%--------------------------------------

By Lemma \ref{lemma1}, one can directly obtain the following results on the
matrices $\bm{B}_{pq}^{n}$ and $\bm{B}_{pq}^{(n,l)}$.

%--------------------------------------
\begin{lemma}[{\cite[Lemma 2.2]{gao2023}}]\label{lemma2}
	Assume that $K_{pq}\in C^{d+1}(D)$ for $p$, $q=1$, $2$ with $d\geq 0$, then
	\begin{align*}
		\bm{B}_{pq}^{n-k+1}=K_{pq}(t_n,t_n)\bm{M}+\bm{O}(h)
		\quad\text{and}\quad
		\bm{B}_{pq}^{(n,n-k)}=K_{pq}(t_n,t_n)\bm{N}+\bm{O}(h)
		\quad\text{for $k=1$, $2$.}
	\end{align*}
	Further, if $d \geq 2$, then
	\[
	\bm{B}_{pq}^{(n,l)}=K_{pq}(t_n,t_l)\bm{N}+\frac{h}{6}\bm{D}_{pq}^{(n,l)}
	+\frac{h^2}{12}\,\tilde{\bm{D}}_{pq}^{(n,l)}+\bm{O}(h^3)
	\quad\text{for $l=0$, \dots, $n-1$.}
	\]
\end{lemma}
%--------------------------------------

We now show the existence and uniqueness of the DG solution for the index-2
IAE~\eqref{IAE}.
%--------------------------------------
\begin{theorem}\label{thr2}
	Assume that $f_p\in C(I)$ and $K_{pq}\in C^{d+1}(D)$ for $p$, $q=1$, $2$
	with~$d\geq 0$. Then, there exists $\bar{h}>0$ such that for any
	uniform mesh~$I_h$ with step size~$h\in (0, \bar{h})$, each of the linear
	algebraic systems \eqref{algebraic-eq} has a unique solution
	$\left(\begin{smallmatrix}\bm{U}_n^1 \\ \bm{U}_n^2
	\end{smallmatrix}\right)\in\mathbb{R}^{2m}$ for all $n=0$, \dots, $N-1$.
	Therefore, the approximate equations \eqref{DGeq1} and \eqref{DGeq2} define a
	unique DG approximate solution~$(x_{1,h},x_{2,h})^T$ for the index-$2$
	IAE~\eqref{IAE}, with $x_{1,h}$, $x_{2,h}\in S_{m-1}^{(-1)}(I_h)$, and its
	representation on the interval~$\sigma_n$ is given by~\eqref{local-DG}.
\end{theorem}
%--------------------------------------
%--------------------------------------
\begin{proof}
	It is obvious that the determinant of the above coefficient matrix in the
	linear algebraic system~\eqref{algebraic-eq} is
	$-h^{2m} \det(\bm{B}_{21}^n) \det(\bm{B}_{12}^n)$. By Lemma \ref{lemma2}, we
	know that
	\begin{align*}
		\det(\bm{B}_{21}^n) \det(\bm{B}_{12}^n)
		&=\det(K_{21}(t_n, t_n)\bm{M}+\bm{O}(h))
		\det(K_{12}(t_n, t_n)\bm{M}+\bm{O}(h))\\
		&=K_{21}(t_n, t_n)K_{12}(t_n, t_n)\det(\bm{M})^2 + O(h).
	\end{align*}
	Further, by Lemma~\ref{l:M}, we know that the matrix~$\bm{M}$ is
	nonsingular which, together with the condition
	$|K_{21}(t,t)K_{12}(t,t)|\geq k_0>0$, yields that
	$\det(\bm{B}_{21}^n)\det(\bm{B}_{12}^n)\ne0$ for sufficiently small~$h$.
	Therefore, there exists a unique solution $(\bm{U}^1_n,
	\bm{U}^2_n)^T\in\mathbb{R}^{2m} $ for all $n=0$, \dots, $N-1$ and
	for sufficiently small~$h$, as claimed.
\end{proof}

%----------------------------------------------------
\section{The perturbed DG method for first-kind VIEs}\label{sec:3}
\subsection{The convergence analysis of the perturbed DG method for
	first-kind VIEs}\label{subsec:3.1}
%----------------------------------------------------
We first consider the first-kind VIE
\begin{equation}\label{V1}
	\int_{0}^{t}k(t,s)y(s)\,\ud s=g(t)\quad\text{for $t\in I$,}
\end{equation}
where~$g$ is a given function satisfying~$g(0)=0$, and $k\in C^2(D)$
with~$|k(t,t)|\geq k_0>0$ for all $t\in I$. Then we consider the
perturbed equation arising when $g$ is replaced by $g+\delta$
with~$\delta=O(h^{m_1})$.  The DG solution~$y_h\in S_{m-1}^{(-1)}(I_h)$ of this
perturbed equation satisfies
\begin{equation}\label{yh}
	\int_{t_n}^{t_{n+1}}\left(\int_{0}^{s}k(s,\tau)y_h(\tau)\mathrm{d}\tau\right)
	\phi(s)\,\ud s
	=\int_{t_n}^{t_{n+1}}g(s)\phi(s)\,\ud s
	+\int_{t_n}^{t_{n+1}}\delta(s)\phi(s)\,\ud s,
\end{equation}
for all $\phi \in S_{m-1}^{(-1)}(I_h)$.

In order to investigate the convergence of~$y_h$, we first give some helpful
lemmas. The first four lemmas come from Brunner et al.~\cite{DG2009},
the fifth one is from our earlier paper~\cite{gao2023}, and the last one can be
obtained by a simple calculation based on Lemma~\ref{lemma1}.

%----------------------------------------------------
\begin{lemma}[{\cite[Lemma 2.3]{DG2009}}] \label{lemma6}
	If $y\in C^{p+1}(I)$, then
	\begin{equation}\label{local-y}
		y(t_n+sh)=\sum_{j=0}^{p}P_j(s)Y^n_j+h^{p+1}R_{p,n}(s)\quad
		\text{for $s\in(0,1]$ and $0\leq n\leq N-1$,}
	\end{equation}
	where $Y_j^n:=\int_{0}^{1}P_j(s)y(t_n+sh)\,\ud s
	/\int_{0}^{1} P_j(s)^2\,\ud s$, and the remainder terms are bounded as follows:
	\[
	|R_{p,n}(s)|\leq\frac{|y^{(p+1)}(\xi_n)|}{(p+1)!}
	\quad\text{for some $\xi_n\in(t_n,t_{n+1})$.}
	\]
	Moreover, the Legendre coefficients~$Y_j^n$ satisfy
	\[
	|Y_j^n -\gamma_j y^{(j)}(t_{n+1/2})h^j|
	\leq\frac{h^{M_{p,j}}(2j+1)}{2^{M_{p,j}}M_{p,j}!}\,\|y^{(M_{p,j})}\|_\infty,
	\]
	with $\gamma_j :=j!/(2j)!$ for $j \geq 0$~and $M_{p,j} := \min\{j+2, p+1\}$.
\end{lemma}
%----------------------------------------------------
It is obvious that for all $\phi\in S_{m-1}^{(-1)}(I_h)$, the exact
solution~$y$ of~\eqref{V1} satisfies
\begin{equation}\label{y}
	\int_{t_n}^{t_{n+1}}\left(\int_{0}^{s}k(s,\tau)y(\tau)\,\ud\tau\right)
	\phi(s)\,\ud s=\int_{t_n}^{t_{n+1}}g(s)\phi(s)\,\ud s.
\end{equation}
Since $y_h\in S_{m-1}^{(-1)}(I_h)$, it has the local representation
\begin{equation}\label{local_yh}
	y_{h}(t_n+sh)=\sum_{j=0}^{m-1}P_j(s)U_{j}^n\quad\text{for $s\in(0,1]$,}
\end{equation}
with unknowns~$U_j^n$. Let $e_h := y -y_{h}$. By \eqref{local-y}~and
\eqref{local_yh},
\begin{equation}\label{general-local-err}
	e(t_n+sh)=\sum_{j=0}^{m-1}P_j(s)E_j^n+\sum_{j=m}^{p}P_j(s)Y_j^n
	+h^{p+1}R_{p,n}(s)\quad\text{for $s\in(0, 1]$,}
\end{equation}
where $E_j^n:=Y_j^n-U_j^n$. In the special case $p=m$, by Lemma \ref{lemma6}, we
have
\begin{equation}\label{special-local-err}
	\begin{aligned}
		e(t_n+sh)&=\sum_{j=0}^{m-1}P_j(s)E_j^n+P_m(s)Y_m^n+h^{m+1}R_{m,n}(s)\\
		&=\sum_{j=0}^{m-1}P_j(s)E_j^n+\gamma_my^{(m)}(t_{n+1/2})h^mP_m(s)+O(h^{m+1}).
	\end{aligned}
\end{equation}
In addition, by \eqref{yh}~and \eqref{y}, the error equations can be written as
\begin{multline*}
	h\int_{0}^{1}\left(\int_{0}^{s}
	k(t_n+sh,t_n+\tau h)e_h(t_n+\tau h)\,\ud\tau\right)
	\phi(t_n+sh)\,\ud s\\
	+h\sum_{l=0}^{n-1}\int_{0}^{1}\left(\int_{0}^{1}
	k(t_n+sh,t_l+\tau h)e_h(t_l+\tau h)\,\ud\tau\right)
	\phi(t_n+sh)\,\ud s\\
	=-\int_{0}^{1}\delta(t_n+sh)\phi(t_n+sh)\,\ud s.
\end{multline*}
Substitute \eqref{general-local-err} into the above equation, then
\begin{equation}\label{general_algebraic_err}
	\sum_{j=0}^{m-1}\beta^n_{i,j}E_j^n
	=-\sum_{l=0}^{n-1}\sum_{j=0}^{m-1}\beta^{(n,l)}_{i,j}E_j^l
	+b^n_i+\frac{\delta_{n,i}}{h}\quad\text{for $i=0$, \dots, $m-1$,}
\end{equation}
where
\begin{gather*}
	\beta^n_{i,j}:=\int_{0}^{1}\left(\int_{0}^{s}
	k(t_n+sh,t_n+\tau h)P_j(\tau)\,\ud\tau\right)P_i(s)\,\ud s,\qquad
	\delta_{n,i}:=-\int_{0}^{1}\delta(t_n+sh)P_i(s)\,\ud s,\\
	\beta^{(n,l)}_{i,j}:=\int_{0}^{1}\left(\int_{0}^{1}
	k(t_n+sh,t_l+\tau h)P_j(\tau)\,\ud\tau\right)P_i(s)\,\ud s
\end{gather*}
and
\begin{align*}
	b_i^n&:=-h^{p+1}\int_{0}^{1}\left(\int_{0}^{s}
	k(t_n+sh,t_n+\tau h)R_{p,n}(\tau)\,\ud\tau\right)P_i(s)\,\ud s \\
	&\qquad{}-\sum_{l=0}^{n-1}h^{p+1}\int_{0}^{1}\left(\int_{0}^{1}
	k(t_n+sh,t_l+\tau h)R_{p,l}(\tau)\,\ud\tau\right)P_i(s)\ud s
	-\sum_{j=m}^{p}\beta^n_{i,j}Y_j^n
	-\sum_{j=m}^{p}\sum_{l=0}^{n-1}\beta^{(n,l)}_{i,j}Y_j^l.
\end{align*}
The above form is used in ours superconvergence analysis, and the following
form for the special case $p=m$ is used in our global convergence analysis:
\begin{equation}\label{special_algebraic_err}
	\sum_{j=0}^{m-1}\beta^n_{i,j}E_j^n
	=-\sum_{l=0}^{n-1}\sum_{j=0}^{m-1}\beta^{(n,l)}_{i,j}E_j^l
	+h^m\rho^n_i+\frac{\delta_{n,i}}{h}+O(h^{m+1}),
\end{equation}
where $\rho^n_i:=-\beta^n_{i,m}\hat{Y}_m^n
-\sum_{l=0}^{n-1}\beta^{(n,l)}_{i,m}\hat{Y}_m^l$ is bounded using
Lemma~\ref{lemma1}, and $\hat{Y}_m^k:=\gamma_my^{(m)}(t_{k+1/2})$
for~$k\geq 0$.

%----------------------------------------------------
\begin{lemma}[{\cite[Lemma 2.4]{DG2009}}]\label{lemma7}
	If $g\in C^{p+2}(I)$ and $k\in C^{p+2}(D)$ with~$p\geq m$, then
	\[
	b_i^n=-\sum_{j=m}^{p-1}\gamma_jh^j\left(O(h^2)
	+\beta^n_{i,j}y^{(j)}(t_{n+1/2})
	+\sum_{l=0}^{n-1}\beta^{(n,l)}_{i,j}y^{(j)}(t_{l+1/2})\right)+O(h^p)
	\quad\text{for $i$, $j\geq 0$,}
	\]
	where $\gamma_j$ is defined in Lemma~\ref{lemma6}.
\end{lemma}
%----------------------------------------------------
\begin{lemma}[{\cite[Lemmas 2.2 and 3.1]{DG2009}}]\label{lemma3}
	Define the $m$-dimensional column vectors
	\begin{align*}
		\bm{e}_i&:=(0,\ldots,0,\underbrace{1}_{i},0,\ldots,0)^T,\\
		\bm{w}&:= 2(2m-1)(1,3,5,\ldots,2m-3,2m-1)^T,\\
		\bm{v}&:=\begin{cases}
			2(1,0,5,\ldots,0,2m-1)^T&\text{if $m$ is odd},\\
			-2(0,3,0,7,\dots,0,2m-1)^T&\text{if $m$ is even},
		\end{cases}
	\end{align*}
	and the $m$-dimensional row vector
	\[
	\bm{\tilde{v}}:=\begin{cases}
		2(1,0,5,\ldots,0,2m-1)&\text{if $m$ is odd}, \\
		2(0,3,0,7,\ldots,0,2m-1)&\text{if $m$ is even}.
	\end{cases}
	\]
	Then,
	\begin{gather*}
		\bm{M}^{-1}\bm{e}_0=\bm{v},\qquad
		\bm{M}^{-1}\bm{e}_1=6\bm{e}_0-3\bm{v},\qquad
		\bm{M}^{-1}\bm{e}_{m-1}=\bm{w},\qquad
		\bm{e}_0^T\bm{M}^{-1}=\bm{\tilde{v}},\\
		\bm{M}^{-1}\bm{N}=\left(\bm{v},\bm{0},\dots,\bm{0}\right),\qquad
		\bm{N}\bm{M}^{-1}=\begin{pmatrix}
			\bm{\tilde{v}} \\ \bm{0}\\ \vdots \\ \bm{0} \\ \end{pmatrix},\qquad
	\end{gather*}
	and moreover $\bm{M}^{-1}\bm{N}\bm{z}=z_0\bm{v}$ and
	$\bm{N}\bm{z}=z_0\bm{e}_0$ for any
	$\bm{z}=(z_0, z_1, \ldots, z_{m-1})^T\in\mathbb{R}^m$.
\end{lemma}
%----------------------------------------------------
\begin{lemma}[{\cite[Lemma 3.1]{DG2009}}]\label{lemma4}
	When $m$ is odd,
	\[
	\bm{G}:=-2\bm{M}^{-1}\bm{N}+(\bm{M}^{-1}\bm{N})^2=\bm{0}
	\quad\text{and}\quad
	~(\bm{I}_m-\bm{M}^{-1}\bm{N})^2=\bm{I}_m.
	\]
\end{lemma}
%-------------------------------------------------
\begin{lemma}[{\cite[Lemma 3.3]{gao2023}}]\label{lemma5}
	When $m$ is even,
	\[
	\bm{N}\bm{M}^{-1}\bm{N}=\bm{0}\quad\text{and}\quad
	\bm{N}\bm{M}^{-1}\bm{D}_{pq}^{(n,l)}=6[\partial_1K_{pq}(t_n,t_l)]\bm{N}
	\quad\text{for $0\leq l \leq n \leq N-1$.}
	\]
\end{lemma}
%--------------------------------------------------
\begin{lemma}\label{lemma8}
	Assume that $k\in C^{d+1}(D)$ with $d\geq 0$, then for all $i$, $j\geq 0$ and
	$0 \leq n \leq N-1$,
	\[
	\beta_{ij}^{n}-\beta_{ij}^{n-1}=\begin{cases}
		hk'(t_n,t_n)/2+O(h^2)&\text{if $(i,j)=(0,0)$,}\\
		-h\alpha_jk'(t_n,t_n)+O(h^2)&\text{if $i=j-1$ and $j\geq1$,}\\
		h\alpha_ik'(t_n,t_n)+O(h^2)&\text{if $j=i-1$ and $i\geq 1$,}\\
		O(h^2)&\text{if $i=j>0$,}\\
		O(h^{|i-j|})+O(h^{d+2})&\text{otherwise},
	\end{cases}
	\]
	where $\alpha_j$ is defined in Lemma~\ref{lemma1}. Further, if $d \geq 2$,
	then for $l=0$, \dots, $n-1$,
	\[
	\beta_{ij}^{{(n,l)}}-\beta_{ij}^{{(n-1,l)}}=\begin{cases}
		h\partial_1k(t_n,t_l)+h^2[\partial_1\partial_2k(t_n,t_l)]/2+O(h^3),
		&\text{if $(i,j)=(0,0)$,}\\
		h^2[\partial_1\partial_{j+1}k(t_n,t_l)]/6+O(h^3),
		&\text{if $(i,j)\in\{(0,1), (1,0)\}$,}\\
		O(h^3)&\text{if $(i,j)\in\{(0,2), (2,0), (1,1)\}$,}\\
		O(h^{i+j+1})+O(h^{d+2})&\text{otherwise}.
	\end{cases}
	\]
\end{lemma}
%--------------------------------------------------
Now we state the main convergence result of the perturbed DG solution for
the first-kind VIE~\eqref{V1}.

%----------------------------------------------------
\begin{theorem}\label{perturbed}
	Assume that $g\in C^{m+3}(I)$ and $k\in C^{m+3}(D)$, and let $y_h\in
	S_{m-1}^{(-1)}(I_h)$ be the perturbed DG solution for~\eqref{V1} defined
	by~\eqref{yh} with~$\delta=O(h^{m_1})$. Then, for all uniform meshes~$I_h$
	with sufficiently small~$h$, $y_h$ converges uniformly to the solution~$y$
	of~\eqref{V1}. The corresponding attainable global order of convergence is given
	by
	\[
	\|y-y_h\|_\infty:=\sup_{t\in I}\left|y(t)-y_h(t)\right|\leq\begin{cases}
		O(h^{\min\{m, m_1-2\}})&\text{if $m$ is odd},\\
		O(h^{\min\{m-1, m_1-2\}})&\text{if $m$ is even}.
	\end{cases}
	\]
\end{theorem}
%----------------------------------------------------
\begin{proof}
	We divide the proof into two cases.
	
	\noindent\textbf{Case I: $m$ is odd.}
	Using standard techniques for the error estimation of VIEs~\cite{Brunner}, and
	rewriting \eqref{special_algebraic_err} with~$n$ replaced by~$n-1$, then
	subtracting it from \eqref{special_algebraic_err}, we have
	\[
	\sum_{j=0}^{m-1}\beta^n_{i,j}E_j^n
	=\sum_{j=0}^{m-1}[\beta^{n-1}_{i,j}-\beta^{(n,n-1)}_{i,j}]E_j^{n-1}
	-h\sum_{l=0}^{n-2}\sum_{j=0}^{m-1}\hat{\beta}^{(n,l)}_{i,j}E_j^l
	+h^{m+1}\hat{\rho}^n_i+\hat{\delta}_{n,i}+O(h^{m+1}),
	\]
	for $i$,$j=0$, $1$, \dots, $m-1$, where
	\[
	\hat{\beta}^{(n,l)}_{i,j}
	:=\tfrac{1}{h}[\beta_{i,j}^{(n,l)}-\beta_{i,j}^{(n-1,l)}], \quad
	\hat{\delta}_{n,i}:=\tfrac{1}{h}(\delta_{n,i}-\delta_{n-1,i}),\quad
	\hat{\rho}^n_i:=\tfrac{1}{h}(\rho^n_i-\rho^{n-1}_i).
	\]
	Define the $m$-dimensional vectors
	\[
	\bm{E}_n:=(E^n_0, E^n_1, \dots, E^n_{m-1})^{T}, \quad
	\bm{\hat{\rho}}_{n}
	:=(\hat{\rho}^n_0, \hat{\rho}^n_1, \dots, \hat{\rho}^n_{m-1})^{T}, \quad
	\bm{\hat{\delta}}_{n}
	:=(\hat{\delta}_{n,0}, \hat{\delta}_{n,1}, \dots, \hat{\delta}_{n,m-1})^{T},
	\]
	and the $m\times m$ matrices
	\begin{gather*}
		\bm{B}^n:=(\beta^n_{i,j}),\quad
		\bm{B}^{(n,l)}:=(\beta^{(n,l)}_{i,j}),\quad
		\bm{Q}^{(n,l)}:=(\beta^{l}_{i,j}-\beta^{(n,l)}_{i,j})
		=\bm{B}^{l}-\bm{B}^{(n,l)},\\
		\bm{W}^{(n,l)}:=(-\hat{\beta}^{(n,l)}_{i,j})
		= -\tfrac{1}{h}[\bm{B}^{(n,l)}-\bm{B}^{(n-1,l)}],
	\end{gather*}
	with indices $i$, $j=0$, $1$, \dots, $m-1$. We then have
	\[
	\bm{B}^n\bm{E}_n=\bm{Q}^{(n,n-1)}\bm{E}_{n-1}
	+h\sum_{l=0}^{n-2}\bm{W}^{(n,l)}\bm{E}_l
	+h^{m+1}\bm{\hat{\rho}}_{n}+\hat{\bm{\delta}}_n+\bm{O}(h^{m+1}).
	\]
	By Lemma~\ref{lemma2}, we know that $\bm{B}^n=(k(t_n,t_n)+O(h))\bm{M}$;
	further, by Lemma~\ref{l:M}, the matrix~$\bm{M}$ is nonsingular, so for
	sufficiently small $h$, the matrix $\bm{B}^n$ is also nonsingular with the
	inverse
	\begin{equation}\label{invBn}
		(\bm{B}^n)^{-1}=\tfrac{1}{k(t_n,t_n)}\,\bm{M}^{-1}+\bm{O}(h).
	\end{equation}
	By Lemmas \ref{lemma2}~and \ref{lemma8}, and by Taylor expansion,
	\begin{equation}\label{e:QW}
		\bm{Q}^{(n,n-1)}=k(t_n,t_n)\left(\bm{M}-\bm{N}\right)+\bm{O}(h)
		\quad\text{and}\quad
		\bm{W}^{(n,l)} = -\partial_1k(t_n,t_l)\bm{N}+\bm{O}(h).
	\end{equation}
	Therefore,
	\begin{equation}\label{original-K-odd}
		\bm{E}_n=\bigl[\bm{H}+\bm{O}(h)\bigr]\bm{E}_{n-1}
		+h\sum_{l=0}^{n-2}\bigl[\bm{H}^{(n,l)}+\bm{O}(h)\bigr]\bm{E}_l
		+\bm{O}(h^{\nu_1}),
	\end{equation}
	where $\nu_1:=\min\{m+1, m_1-1\}$ with
	\[
	\bm{H} := \bm{I}_m-\bm{M}^{-1}\bm{N}
	\quad\text{and}\quad
	\bm{H}^{(n,l)}:=-\tfrac{\partial_{1}k(t_n,t_l)}{k(t_n,t_n)}\bm{M}^{-1}\bm{N}.
	\]
	Using $\eqref{original-K-odd}_n+\left[\bm{H}+\bm{O}(h)\right]
	\eqref{original-K-odd}_{n-1}$ and Taylor expansion, we have
	\[
	\bm{E}_n=\bigl[\bm{\bar{H}}+\bm{O}(h)\bigr]\bm{E}_{n-2}
	+h\sum_{l=0}^{n-3}\bigl[\bm{\bar{H}}^{(n,l)}+\bm{O}(h)\bigr]\bm{E}_l
	+\bm{O}(h^{\nu_1}),
	\]
	where
	\[
	\bm{\bar{H}}:= \bm{H}^2 = (\bm{I}_m-\bm{M}^{-1}\bm{N})^2=\bm{I}_m
	\quad\text{and}\quad
	\bm{\bar{H}}^{(n,l)} :=  -\tfrac{\partial_{1}k(t_n,t_l)}{k(t_n,t_n)} \bm{G}
	=\bm{0}_{m}.
	\]
	Here we have used Lemma~\ref{lemma4}. Therefore, the above system can be
	written as
	\begin{equation}\label{global_odd}
		\bm{E}_n=\left[\bm{I}_{m}+\bm{O}(h)\right]\bm{E}_{n-2}
		+\sum_{l=0}^{n-3}\bm{O}(h^2)\bm{E}_l+\bm{O}(h^{\nu_1}),
	\end{equation}
	so by induction,
	\begin{align*}
		\|\bm{E}_n\|_\infty&\leq(1+Ch)\|\bm{E}_{n-2}\|_\infty
		+Ch^2\sum_{l=0}^{n-3}\|\bm{E}_l\|_\infty+Ch^{\nu_1}\\
		& \leq (1+Ch)^{n-2}\|\bm{E}_0\|_\infty
		+Ch^2\sum_{l=0}^{n-3}\left[\sum_{k=0}^{n-3-l}(1+Ch)^k \right]
		\|\bm{E}_l\|_\infty
		+Ch^{\nu_1}\left[\sum_{k=0}^{n-3}(1+Ch)^k\right]\\
		&\leq Ch\sum_{l=0}^{n-3}\|\bm{E}_l\|_\infty+Ch^{\nu_1-1}.
	\end{align*}
	By the discrete Gronwall inequality~\cite[Corollary~2.1.19]{Brunner},
	we have $\|\bm{E}_n\|_\infty\leq Ch^{\nu_1-1}$, which together with
	\eqref{special-local-err} yields the desired result.
	
	\medskip
	
	\noindent{\textbf{Case II: $m$ is even.}} By~\eqref{special_algebraic_err},
	\begin{equation}\label{generalKeven}
		\bm{B}^n\bm{E}_n=-\sum_{l=0}^{n-1}\bm{B}^{(n,l)}\bm{E}_l+h^{m}\bm{\rho}^n
		+\tilde{\bm{\delta}}_n+\bm{O}(h^{m+1}),
	\end{equation}
	where
	\[
	\tilde{\delta}_{n,i}:=\tfrac{1}{h}\delta_{n,i}, \quad
	\bm{\rho}^n:=(\rho^n_0,\rho^n_1,\dots, \rho^n_{m-1})^T, \quad
	\bm{\tilde{\delta}}_n:=(\tilde{\delta}_{n,0},\tilde{\delta}_{n,1},\dots,
	\tilde{\delta}_{n,m-1})^T.
	\]
	Using $\eqref{generalKeven}_n-\bm{F}\eqref{generalKeven}_{n-1}$, with
	$\bm{F}:=\bm{I}_m+\bm{N}\bm{M}^{-1}$, we have
	\begin{equation}\label{even-al-err-eq}
		\bm{B}^n\bm{E}_n=\bm{\tilde{Q}}^n\bm{E}_{n-1}
		+\sum_{l=0}^{n-2}\bm{\tilde{Q}}^{(n,l)}\bm{E}_l+\bm{O}(h^{\nu_2}),
	\end{equation}
	where $\nu_2:=\min\{m, m_1-1\}$ with
	\begin{equation}\label{e:Q}
		\bm{\tilde{Q}}^n:=\bm{F}\bm{B}^{n-1}-\bm{B}^{(n,n-1)}
		\quad\text{and}\quad
		\bm{\tilde{Q}}^{(n,l)}:=\bm{F}\bm{B}^{(n-1,l)}-\bm{B}^{(n,l)}.
	\end{equation}
	By Lemma~\ref{lemma2}, we have
	\[
	\bm{\tilde{Q}}^n
	=(\bm{I}_m+\bm{N}\bm{M}^{-1})\bm{B}^{n-1}-\bm{B}^{(n,n-1)}
	=k(t_n,t_n)\bm{M}+\bm{O}(h)
	=\bm{B}^n + \bm{O}(h),
	\]
	and by Taylor expansion, Lemmas \ref{lemma2}~and \ref{lemma5}, and
	\eqref{e:QW},
	\begin{align*}
		\bm{\tilde{Q}}^{(n,l)}
		&=(\bm{I}_m+\bm{N}\bm{M}^{-1})\bm{B}^{(n-1,l)}-\bm{B}^{(n,l)}
		=\bm{B}^{(n-1,l)}-\bm{B}^{(n,l)}+\bm{N}\bm{M}^{-1}\bm{B}^{(n-1,l)}  \\
		&=-h\bm{W}^{(n,l)}+\bm{N}\bm{M}^{-1}\left[k(t_{n-1},t_l)\bm{N}
		+\tfrac{h}{6}\bm{D}^{(n-1,l)}+\bm{O}(h^2)\right]  \\
		&=-h^2[\partial_1^2 k(\cdot,t_l)]\bm{N}+\bm{O}(h^2)=\bm{O}(h^2).
	\end{align*}
	Then, by \eqref{invBn},
	\[
	\bm{E}_n=\left[\bm{I}_{m}+\bm{O}(h)\right]\bm{E}_{n-1}
	+\sum_{l=0}^{n-2}\bm{O}(h^2)\bm{E}_l+\bm{O}(h^{\nu_2}),
	\]
	and by induction,
	\begin{align*}
		\|\bm{E}_n\|_\infty &\leq (1+Ch)\|\bm{E}_{n-1}\|_\infty
		+Ch^2\sum_{l=0}^{n-2}\|\bm{E}_l\|_\infty+Ch^{\nu_2}\\
		&\leq(1+Ch)^{n}\|\bm{E}_0\|_{\infty}
		+Ch^2\sum_{l=0}^{n-2}\left[\sum_{k=0}^{n-2-l}(1+Ch)^k\right]\|\bm{E}_l\|_\infty
		+Ch^{\nu_2}\left[\sum_{k=0}^{n-1}(1+Ch)^k \right] \\
		&\leq Ch\sum_{l=0}^{n-2}\|\bm{E}_l\|_\infty+Ch^{\nu_2-1}.
	\end{align*}
	Finally, by the discrete Gronwall inequality~\cite[Corollary 2.1.19]{Brunner},
	we obtain $\|\bm{E}_n\|_\infty\leq Ch^{\nu_2-1}$, which together
	with~\eqref{special-local-err} yields the desired result.
\end{proof}
%-------------------------------------------------------

%-------------------------------------------------------
\subsection{Supercovergence analysis of the perturbed DG method for first-kind
	VIEs}\label{subsec:3.2}

First, we give some useful lemmas. For the sake of simplicity,
set $k'_{n,l}:=k'(t_{n},t_{l})/k(t_{n},t_{n})$ for $0 \leq l \leq n \leq N$.

%----------------------------------------------------			
\begin{lemma}\label{lemmaLodd}
	Define the vector~$\bm{q}:=\bm{w}/w_0-\bm{v}/2$, and assume that the matrices
	$\mathbb{A}_n^0$ and $\mathbb{A}_n^l$ are bounded as~$h\to0$.  If
	\[
	\mathbb{V}_n^0:=\mathbb{A}_n^0\bm{M}^{-1}\bm{N}-2k'_{n-1,n-1}\bm{I}_m+\bm{O}(h)
	\quad\text{and}\quad
	\mathbb{V}_n^{l}:=\mathbb{A}_n^l\bm{M}^{-1}\bm{N}+\bm{O}(h)
	\quad\text{for $l\geq 1$.}
	\]
	and if $m$ is odd, then
	\begin{align}
		\mathbb{V}_n^0\bm{q}=-2k'_{n-1,n-1}\bm{q} + \bm{O}(h), \qquad
		\mathbb{V}_n^l\bm{q}= \bm{O}(h)\quad\text{for all $l\geq 1$,}\label{e:001}\\
		(\bm{B}^n)^{-1}\bm{Q}^{(n,n-1)}(\bm{B}^{n-1})^{-1}\bm{Q}^{(n-1,n-2)}
		-h(\bm{B}^n)^{-1}\bm{W}^{(n,n-2)}=\bm{I}_m+h\mathbb{V}_n^0,\label{e:1}\\
		(\bm{B}^n)^{-1}\bm{W}^{(n,l)}
		+(\bm{B}^n)^{-1}\bm{Q}^{(n,n-1)}(\bm{B}^{n-1})^{-1}
		\bm{W}^{(n-1,l)}=h\mathbb{V}_n^{l}\quad\text{for $0\leq l\leq n\leq N-1$.}
		\label{e:2}
	\end{align}
\end{lemma}
\begin{proof}
	The formula~\eqref{e:001} can be easily checked by noticing that $q_0=0$ and
	then, by Lemma \ref{lemma3}, $\bm{M}^{-1}\bm{N}\bm{q}=q_0\bm{v}=\bm{0}$.
	In the following, we prove \eqref{e:1}~and \eqref{e:2}.
	
	By Lemma~\ref{lemma1}, we can write
	$\bm{B}^{(n,n-1)}=k(t_n,t_n)\bm{N}+h\mathbb{A}_2^0+\bm{O}(h^2)$ with the bounded
	matrix
	\begin{equation}\label{e:A20}
		\mathbb{A}_2^0 := \begin{pmatrix}
			a_{00}&a_{01}&0     &\cdots&0\\
			a_{10}&0     &0     &\cdots&0\\
			0     &0     &0     &\cdots&0 \\
			\vdots&\vdots&\vdots&\ddots&\vdots \\
			0     &0     &0     &\cdots&0
		\end{pmatrix}
	\end{equation}
	for some constants $a_{00}, a_{01}, a_{10} \in \mathbb{R}$.
	By Lemma~\ref{lemma3}, for any
	$\bm{z}=(z_0,z_1,\dots,z_{m-1})^T\in\mathbb{R}^m$,
	\begin{equation}\label{e:M-1A20}
		\begin{aligned}
			\bm{M}^{-1}\mathbb{A}_2^0\bm{z}
			&=\bm{M}^{-1}[(a_{00}z_0+a_{01}z_1)\bm{e}_0+(a_{10}z_0)\bm{e}_1] \\
			&=6a_{10}z_0\bm{e}_0+(-3a_{10}z_0+a_{00}z_0+a_{01}z_1)\bm{v} \\
			&=6a_{10}\bm{N}\bm{z}
			+[(-3a_{10}z_0+a_{00}z_0+a_{01}z_1)/(z_0)]\bm{M}^{-1}\bm{N}\bm{z},
		\end{aligned}
	\end{equation}
	which means that
	\[
	\bm{M}^{-1}\mathbb{A}_2^0 = \mathbb{A}_0\bm{M}^{-1}\bm{N}
	\]
	for the bounded matrix~$\mathbb{A}_0 := 6a_{10}\bm{M}
	+[(-3a_{10}z_0+a_{00}z_0+a_{01}z_1)/(z_0)]\bm{I}_m$.
	
	In addition, if we assume that
	\begin{equation}\label{eq:Bninv}
		(\bm{B}^n)^{-1}=\tfrac{1}{k(t_n,t_n)}\bm{M}^{-1}+h\mathbb{A}_{0}^n+\bm{O}(h^2)
	\end{equation}
	for some bounded matrix~$\mathbb{A}_{0}^n$, then
	\begin{align}
		(\bm{B}^n)^{-1}\bm{B}^{(n,n-1)}&=\bm{M}^{-1}\bm{N}
		+\tfrac{h}{k(t_n,t_n)}\bm{M}^{-1}\mathbb{A}_{2}^0
		+hk(t_n,t_n)\mathbb{A}_{0}^n\bm{N}+\bm{O}(h^2)\label{e:Bn-for-even}\\
		&=\bm{M}^{-1}\bm{N}+h\tilde{\mathbb{A}}_n^0\bm{M}^{-1}\bm{N}
		+\bm{O}(h^2),\label{e:Bn}
	\end{align}
	where $\tilde{\mathbb{A}}_n^0 :=
	[1/k(t_n,t_n)]\mathbb{A}_{0}+k(t_n,t_n)\mathbb{A}_{0}^n\bm{M}$.
	
	By Lemma \ref{lemma8}, we have
	\begin{equation}\label{eq:Bn-1}
		\bm{B}^{n-1}=\bm{B}^n-hk'(t_n,t_n)\bm{M}+\bm{O}(h^2).
	\end{equation}
	Then
	\[
	\bm{Q}^{(n,n-1)}=\bm{B}^{n-1}-\bm{B}^{(n,n-1)}
	=\bm{B}^n-\bm{B}^{(n,n-1)}-hk'(t_n,t_n)\bm{M}+\bm{O}(h^2),
	\]
	and further, by \eqref{eq:Bninv}~and \eqref{e:Bn},
	\begin{equation}\label{e:BQ}
		\begin{aligned}
			(\bm{B}^n)^{-1}\bm{Q}^{(n,n-1)}&=\bm{I}_m-(\bm{B}^n)^{-1}\bm{B}^{(n,n-1)}
			-hk'_{n,n}\bm{I}_m+\bm{O}(h^2)\\
			&=\bm{I}_m-\bm{M}^{-1}\bm{N}-h\tilde{\mathbb{A}}_n^0\bm{M}^{-1}\bm{N}
			-hk'_{n,n}\bm{I}_m+\bm{O}(h^2);
		\end{aligned}
	\end{equation}
	similarly,
	\[
	(\bm{B}^{n-1})^{-1}\bm{Q}^{(n-1,n-2)}=\bm{I}_m -\bm{M}^{-1}\bm{N}
	-h\tilde{\mathbb{A}}_n^1\bm{M}^{-1}\bm{N}-hk'_{n-1,n-1}\bm{I}_m+\bm{O}(h^2)
	\]
	for some bounded matrix~$\tilde{\mathbb{A}}_n^1$. By Lemma~\ref{lemma4}, we know
	that $\bm{G}=-2\bm{M}^{-1}\bm{N}+(\bm{M}^{-1}\bm{N})^2=\bm{0}$, which implies
	that
	\begin{equation}\label{e:MN}
		-\bm{M}^{-1}\bm{N}+(\bm{M}^{-1}\bm{N})^2 = \bm{M}^{-1}\bm{N}.
	\end{equation}
	Thus, by Taylor expansion and Lemma~\ref{lemma4},
	\begin{align*}
		(\bm{B}^n)^{-1}&\bm{Q}^{(n,n-1)}(\bm{B}^{n-1})^{-1}\bm{Q}^{(n-1,n-2)}\\
		&=(\bm{I}_m-\bm{M}^{-1}\bm{N})^2-h(\bm{I}_m
		-\bm{M}^{-1}\bm{N})\tilde{\mathbb{A}}_n^1\bm{M}^{-1}\bm{N}
		-h\tilde{\mathbb{A}}_n^0\bm{M}^{-1}\bm{N}(\bm{I}_m-\bm{M}^{-1}\bm{N})\\
		&\qquad{}-2h(\bm{I}_m-\bm{M}^{-1}\bm{N})k'_{n-1,n-1}\bm{I}_m+\bm{O}(h^2) \\
		&=\bm{I}_m+h\left[-(\bm{I}_m-\bm{M}^{-1}\bm{N})\tilde{\mathbb{A}}_n^1
		+\tilde{\mathbb{A}}_n^0+2k'_{n-1,n-1}\bm{I}_m\right]\bm{M}^{-1}\bm{N}
		-2hk'_{n-1,n-1}\bm{I}_m+\bm{O}(h^2) \\
		&=:\bm{I}_m+h\widehat{\mathbb{A}}_n^0\bm{M}^{-1}\bm{N}
		-2hk'_{n-1,n-1}\bm{I}_m+\bm{O}(h^2).
	\end{align*}
	By Lemma~\ref{lemma8}, we can write
	\[
	\bm{W}^{(n,l)}=\partial_1k(t_n,t_l)\bm{N}+h\mathbb{A}_{21}^{(n,l)}+\bm{O}(h^2)
	\quad\text{and}\quad
	\bm{W}^{(n-1,l)}=\partial_1k(t_n,t_l)\bm{N}
	+h\mathbb{A}_{22}^{(n,l)}+\bm{O}(h^2),
	\]
	where $\mathbb{A}_{21}^{(n,l)}$ and $\mathbb{A}_{22}^{(n,l)}$ have the same
	structure as $\mathbb{A}_2^0$. Then similarly, we can write
	$\bm{M}^{-1}\mathbb{A}_{21}^{(n,l)} = \mathbb{A}^{(n,l)} \bm{M}^{-1}\bm{N}$ for
	some bounded matrix~$\mathbb{A}^{(n,l)}$, so
	\[
	(\bm{B}^n)^{-1}\bm{W}^{(n,l)}=k'_{n,l}\bm{M}^{-1}\bm{N}
	+h\tilde{\mathbb{A}}_{21}^{(n,l)}\bm{M}^{-1}\bm{N}+\bm{O}(h^2),
	\]
	where $\tilde{\mathbb{A}}_{21}^{(n,l)}:=[1/k(t_n,t_n)]\mathbb{A}^{(n,l)}
	+\partial_1k(t_n,t_l)\mathbb{A}_{0}^n\bm{M}$ is bounded. Similarly, there
	exists a bounded matrix $\tilde{\mathbb{A}}_{22}^{(n,l)}$ such that
	\begin{equation}\label{e:BW}
		(\bm{B}^{n-1})^{-1}\bm{W}^{(n-1,l)}=k'_{n,l}\bm{M}^{-1}\bm{N}
		+h\tilde{\mathbb{A}}_{22}^{(n,l)}\bm{M}^{-1}\bm{N}+\bm{O}(h^2).
	\end{equation}
	Therefore,
	\begin{align*}
		(\bm{B}^n)^{-1}\bm{Q}^{(n,n-1)}&(\bm{B}^{n-1})^{-1}\bm{Q}^{(n-1,n-2)}
		-h(\bm{B}^n)^{-1}\bm{W}^{(n,n-2)} \\
		&=\bm{I}_m+h\left[\widehat{\mathbb{A}}_n^0-k'_{n,n-2}\bm{I}_m
		-\tilde{\mathbb{A}}_{21}^{(n,n-2)}\right]\bm{M}^{-1}\bm{N}
		-2hk'_{n-1,n-1}\bm{I}_m+\bm{O}(h^2)\\
		&=:\bm{I}_m+h\mathbb{A}_n^0\bm{M}^{-1}\bm{N}
		-2hk'_{n-1,n-1}\bm{I}_m+\bm{O}(h^2),
	\end{align*}
	and \eqref{e:1} follows.
	
	Now we prove \eqref{e:2}.
	By \eqref{e:BQ}, \eqref{e:MN}, \eqref{e:BW} and Lemma~\ref{lemma4},
	\begin{align*}
		&(\bm{B}^n)^{-1}\bm{Q}^{(n,n-1)}(\bm{B}^{n-1})^{-1}\bm{W}^{(n-1,l)}\\
		&=\bigl[\bm{I}_m-\bm{M}^{-1}\bm{N}-h\tilde{\mathbb{A}}_n^0\bm{M}^{-1}\bm{N}
		-2hk'_{n-1,n-1}\bm{I}_m+\bm{O}(h^2)\bigr]
		\bigl[k'_{n,l}\bm{M}^{-1}\bm{N}
		+h\tilde{\mathbb{A}}_{22}^{(n,l)}\bm{M}^{-1}\bm{N}+\bm{O}(h^2)\bigr]\\
		&=-k'_{n,l}\bm{M}^{-1}\bm{N}+h\tilde{\mathbb{A}}_n^l\bm{M}^{-1}\bm{N}
		+\bm{O}(h^2),
	\end{align*}
	with the obvious meanings of the bounded matrices $\tilde{\mathbb{A}}_n^l$.
	Therefore,
	\[
	(\bm{B}^n)^{-1}\bm{W}^{(n,l)}
	+(\bm{B}^n)^{-1}\bm{Q}^{(n,n-1)}(\bm{B}^{n-1})^{-1}\bm{W}^{(n-1,l)}
	=h\mathbb{A}_n^l\bm{M}^{-1}\bm{N}+\bm{O}(h^2),
	\]
	with $\mathbb{A}_n^l:=\tilde{\mathbb{A}}_{21}^{(n,l)}+\tilde{\mathbb{A}}_n^l$.
	The proof is complete.
\end{proof}

%-----------------------------------------------------
In the next lemma, we use $\tilde{\bm{Q}}^n$ and $\tilde{\bm{Q}}^{(n,l)}$
defined in~\eqref{e:Q}.

\begin{lemma}\label{lemmaLeven}
	For some matrices~$\mathbb{\bar{A}}_n:=(\bar{a}_{i,j}^n)$ which are bounded
	as~$h\to0$, and for $\mathbb{A}_2^0$ defined in~\eqref{e:A20}, let
	\[
	\mathbb{W}_n:=-k_{n,n}'\bm{I}_m+\bm{M}^{-1}\bm{N}\mathbb{\bar{A}}_n
	-k_{n,n}'\bm{M}^{-1}\bm{N}
	-\tfrac{1}{k(t_n,t_n)}\bm{M}^{-1}\mathbb{A}_2^0+\bm{O}(h).
	\]
	If $m$ is even, then
	\begin{equation}\label{Wnz}
		\mathbb{W}_n\bm{z}=-k_{n,n}'\bm{z}+\bar{a}_n^0\bm{v}+\bar{a}_n^1z_0\bm{e}_0
		+\bm{O}(h)\quad
		\text{for all $\bm{z}=(z_0,z_1,\dots,z_{m-1})^T \in\mathbb{R}^m$,}
	\end{equation}
	with scalars
	$\bar{a}_n^0:=\sum_{j=0}^{m-1}\bar{a}_{0,j}^nz_j-k_{n,n}'z_0+(3a_{10}z_0
	-a_{00}z_0-a_{01}z_1)/k(t_n,t_n)$~and
	$\bar{a}_n^1:=-6a_{10}/k(t_n,t_n)$ that are bounded as~$h\to0$.  Moreover,
	\begin{align}
		(\bm{B}^n)^{-1}\tilde{\bm{Q}}^n&=\bm{I}_m+h\mathbb{W}_n, \label{e:mathbbW}\\
		(\bm{B}^n)^{-1}\tilde{\bm{Q}}^{(n,l)}
		&=-h^2\tfrac{\partial_{1}^2k(t_{n-1/2},t_l)}{k(t_n,t_n)}\bm{N}
		+\bm{O}(h^3)\quad\text{for $0\leq l \leq n \leq N-1$.} \label{e:BQL}
	\end{align}
\end{lemma}
\begin{proof}
	The equation~\eqref{Wnz} can be obtained directly from Lemma~\ref{lemma3}~and
	\eqref{e:M-1A20}. Turning to~\eqref{e:mathbbW}, by \eqref{eq:Bninv}~and
	\eqref{eq:Bn-1}, we have
	\begin{align*}
		(\bm{B}^n)^{-1}\bm{B}^{n-1}
		&=(\bm{B}^n)^{-1}\left[\bm{B}^n-hk'(t_n,t_n)\bm{M}+\bm{O}(h^2)\right] \\
		&=\bm{I}_m+(\bm{B}^n)^{-1}\left[-hk'(t_n,t_n)\bm{M}+\bm{O}(h^2)\right]\\
		&=\bm{I}_m+\left[\tfrac{1}{k(t_n,t_n)}\bm{M}^{-1}+h\mathbb{A}_{0}^n
		+\bm{O}(h^2)\right]
		\left[-hk'(t_n,t_n)\bm{M}+\bm{O}(h^2)\right] \\
		&=\bm{I}_m-hk'_{n,n}\bm{I}_m+\bm{O}(h^2)
	\end{align*}
	and
	\[
	(\bm{B}^n)^{-1}\bm{N}=\left[\tfrac{1}{k(t_n,t_n)}\bm{M}^{-1}+h\mathbb{A}_{0}^n
	+\bm{O}(h^2)\right]\bm{N}
	=\tfrac{1}{k(t_n,t_n)}\bm{M}^{-1}\bm{N}+h\mathbb{A}_{0}^n\bm{N}+\bm{O}(h^2).
	\]
	By Lemma \ref{lemma1}, we can assume that
	$\bm{B}^n=k(t_n,t_n)\bm{M}+h\mathbb{\bar{A}}_n^3+\bm{O}(h^2)$ for some bounded
	matrix~$\mathbb{\bar{A}}_n^3$. Then, 
	\begin{align*}
		(\bm{B}^n)^{-1}\bm{N}\bm{M}^{-1}\bm{B}^n
		&=(\bm{B}^n)^{-1}\bm{N}\bm{M}^{-1}\left[
		k(t_n,t_n)\bm{M}+h\mathbb{\bar{A}}_n^3+\bm{O}(h^2)\right] \\
		&=k(t_n,t_n)(\bm{B}^n)^{-1}\bm{N}
		+h(\bm{B}^n)^{-1}\bm{N}\bm{M}^{-1}\mathbb{\bar{A}}_n^3+\bm{O}(h^2) \\
		\text{and again by \eqref{eq:Bninv},} \\
		&=\bm{M}^{-1}\bm{N}+hk(t_n,t_n)\mathbb{A}_{0}^n\bm{N}+h\left[
		\tfrac{1}{k(t_n,t_n)}\bm{M}^{-1}+h\mathbb{A}_{0}^n+\bm{O}(h^2)
		\right]\bm{N}\bm{M}^{-1}\mathbb{\bar{A}}_n^3+\bm{O}(h^2) \\
		&=\bm{M}^{-1}\bm{N}+hk(t_n,t_n)\mathbb{A}_{0}^n\bm{N}
		+\tfrac{h}{k(t_n,t_n)}\bm{M }^{-1}\bm{N}\bm{M}^{-1}\mathbb{\bar{A}}_n^3
		+\bm{O}(h^2)\\
		&=:\bm{M}^{-1}\bm{N}+hk(t_n,t_n)\mathbb{A}_{0}^n\bm{N}
		+h\bm{M}^{-1}\bm{N}\mathbb{\bar{A}}_n+\bm{O}(h^2).
	\end{align*}
	Again by \eqref{eq:Bn-1}, we have
	\begin{align*}
		(\bm{B}^n)^{-1}\bm{N}\bm{M}^{-1}\bm{B}^{n-1}
		&=(\bm{B}^n)^{-1}\bm{N}\bm{M}^{-1}\left[
		\bm{B}^n-hk'(t_n,t_n)\bm{M}+\bm{O}(h^2)\right] \\
		&=(\bm{B}^n)^{-1}\bm{N}\bm{M}^{-1}\bm{B}^n-hk'(t_n,t_n)(\bm{B}^n)^{-1}\bm{N}
		+\bm{O}(h^2) \\
		&=\bm{M}^{-1}\bm{N}+hk(t_n,t_n)\mathbb{A}_{0}^n\bm{N}
		+h\bm{M}^{-1}\bm{N}\mathbb{\bar{A}}_n-hk_{n,n}'\bm{M}^{-1}\bm{N}+\bm{O}(h^2),
	\end{align*}
	which, together with~\eqref{e:Bn-for-even}, implies
	\begin{align*}
		(\bm{B}^n)^{-1}\tilde{\bm{Q}}^n
		&=(\bm{B}^n)^{-1}(\bm{F}\bm{B}^{n-1}-\bm{B}^{(n,n-1)})
		=(\bm{B}^n)^{-1}\left[
		(\bm{I}_m+\bm{N}\bm{M}^{-1})\bm{B}^{n-1}-\bm{B}^{(n,n-1)}\right] \\
		&=(\bm{B}^n)^{-1}(\bm{I}_m+\bm{N}\bm{M}^{-1})\bm{B}^{n-1}
		-(\bm{B}^n)^{-1}\bm{B}^{(n,n-1)} \\
		&=(\bm{B}^n)^{-1}\bm{B}^{n-1}+(\bm{B}^n)^{-1}\bm{N}\bm{M}^{-1}\bm{B}^{n-1}
		-(\bm{B}^n)^{-1}\bm{B}^{(n,n-1)} \\
		&=\bm{I}_m-hk'_{n,n}\bm{I}_m+h\bm{M}^{-1}\bm{N}\mathbb{\bar{A}}_n
		-hk_{n,n}'\bm{M}^{-1}\bm{N}-\tfrac{h}{k(t_n,t_n)}\bm{M}^{-1}\mathbb{A}_2^0
		+\bm{O}(h^2).
	\end{align*}
	
	Now we focus on~\eqref{e:BQL}. By Lemma~\ref{lemma8} and using Taylor
	expansion at the point~$(t_{n-1/2},t_l)$,
	\[
	\bm{B}^{(n,l)}-\bm{B}^{(n-1,l)}=h\partial_{1}k(t_{n-1/2},t_l)\bm{N}
	+\frac{h^2}{6}\bm{D}_1^{(n,l)}+\bm{O}(h^3),
	\]
	with
	\[
	\bm{D}_1^{(n,l)}:=\begin{pmatrix}
		3\partial_1^2k(t_{n-1/2},t_l)+3\partial_{1}\partial_2k(t_{n-1/2},t_l) &
		\partial_{1}\partial_{2}k(t_{n-1/2},t_l) & \cdots &0\\
		\partial_1^2k(t_{n-1/2},t_l) & 0 &\cdots&0 \\
		\vdots&\vdots&\ddots&\vdots\\
		0&0&\cdots&0 \end{pmatrix};
	\]
	In addition, by Lemmas \ref{lemma2}~and \ref{lemma5},
	\begin{align*}
		\bm{N}\bm{M}^{-1}\bm{B}^{(n-1,l)}&=\bm{N}\bm{M}^{-1}\left[k(t_{n-1},t_l)\bm{N}
		+\frac{h}{6}\bm{D}^{(n-1,l)}+\frac{h^2}{12}\tilde{D}^{(n-1,l)}
		+\bm{O}(h^3)\right]\\
		&=h[\partial_{1}k(t_{n-1},t_l)]\bm{N}
		+\frac{h^2}{12}\bm{N}\bm{M}^{-1}\tilde{\bm{D}}^{(n-1,l)}+\bm{O}(h^3),
	\end{align*}
	and by Lemma \ref{lemma3},
	\[
	\bm{N}\bm{M}^{-1}\tilde{\bm{D}}^{(n-1,l)}=\begin{pmatrix}
		6d_{1,0}^{(n-1,l)}&6d_{1,1}^{(n-1,l)}&0     &\cdots&0 \\
		0                 &0                 &0     &\cdots&0 \\
		\vdots            &\vdots            &\vdots&\ddots&\vdots\\
		0                 &0                 &0     &\cdots&0\end{pmatrix};
	\]
	Therefore,
	\begin{align*}
		\tilde{\bm{Q}}^{(n,l)}&=\bm{B}^{(n-1,l)}-\bm{B}^{(n,l)}
		+\bm{N}\bm{M}^{-1}\bm{B}^{(n-1,l)}\\
		&=-\frac{h^2}{2}[\partial_{1}^2k(\vartheta_n,t_l)]\bm{N}
		-\frac{h^2}{6}\bm{D}_1^{(n,l)}
		+\frac{h^2}{12}\bm{N}\bm{M}^{-1}\tilde{\bm{D}}^{(n-1,l)}+\bm{O}(h^3)
		=h^2\bm{H}_{n,l}+\bm{O}(h^3)
	\end{align*}
	for some~$\vartheta_n \in (t_{n-1/2}, t_{n-1})$ and for
	\[
	\bm{H}_{n,l}:= \begin{pmatrix}
		-\frac{1}{2}\partial_{1}^2k(\vartheta_n,t_l)&0     &0     &\cdots&0 \\
		-\frac{1}{6}\partial_{1}^2k(t_{n-1/2},t_l)  &0     &0     &\cdots&0 \\
		0                                           &0     &0     &\cdots&0\\
		\vdots                                      &\vdots&\vdots&\ddots&\vdots\\
		0                                           &0     &0     &\cdots& 0
	\end{pmatrix}=\left(\chi, \bm{0}, \ldots, \bm{0}\right),
	\]
	where $\chi:=\bigl[-\frac{1}{2}\partial_{1}^2k(\vartheta_n,t_l)\bm{e}_0
	-\frac{1}{6}\partial_{1}^2k(t_{n-1/2},t_l)\bm{e}_1\bigr]$.
	Then by~\eqref{eq:Bninv},
	\[
	(\bm{B}^n)^{-1}\tilde{\bm{Q}}^{(n,l)}
	=\frac{h^2}{k(t_n,t_n)}\,\bm{M}^{-1}\bm{H}_{n,l}+\bm{O}(h^3)
	=\frac{h^2}{k(t_n,t_n)}\bm{M}^{-1}
	\left(\chi, \bm{0}, \dots, \bm{0} \right)+\bm{O}(h^3).
	\]
	In addition, by Lemma~\ref{lemma3},
	\[
	\bm{M}^{-1}\chi=-\partial_{1}^2k(t_{n-1/2},t_l)\bm{e}_0+\bm{O}(h),
	\]
	and we obtain
	\[
	(\bm{B}^n)^{-1}\tilde{\bm{Q}}^{(n,l)}
	=-h^2\tfrac{\partial_{1}^2k(t_{n-1/2},t_l)}{k(t_n,t_n)}\bm{N}+\bm{O}(h^3).
	\]
\end{proof}
%-----------------------------------------------------

In the next theorem, we state a local superconvergence result for the perturbed
DG solution of a first-kind VIE.  The special superconvergence points depend on
the zeros of a shifted Legendre polynomial.

\begin{theorem}\label{sup_perturbed}
	Let $y$ be the exact solution of the first-kind VIE~\eqref{V1}, and
	$y_h\in S_{m-1}^{(-1)}(I_h)$ be the corresponding perturbed DG solution defined
	by~\eqref{yh} with $\delta=O(h^{m_1})$.
	\begin{enumerate}
		\item Suppose that $m$ is odd, $g\in C^{m+4}(I)$, $k\in C^{m+4}(D)$, and
		the exact solution $y$ satisfies $y^{(m)}(0)=0$. Then, for all uniform
		meshes~$I_h$ with sufficiently small~$h$,
		\[
		|y(t_n+s_rh)-y_h(t_n+s_rh)|\leq Ch^{\min\{m+1, m_1-2\}},
		\]
		where the points $\{s_r\}_{r=0}^{m-1}$ satisfy $P_{m+1}'(s_r)=0$.
		\item Suppose that $m$ is even, $g\in C^{m+3}(I)$, and $k\in C^{m+3}(D)$.
		Then,
		\[
		|y(t_n+s_rh)-y_h(t_n+s_rh)| \leq Ch^{\min\{m, m_1-2\}},
		\]
		where the points $\{s_r\}_{r=1}^{m-1}$ satisfy $P_{m}'(s_r)=0$.
	\end{enumerate}
\end{theorem}

\begin{proof}
	We again divide the proof into two cases.
	
	\noindent{\textbf{Case I: \bm{$m$} is odd.}}
	Set $\bm{b}^n:=(b^n_0, b^n_1, \dots, b^n_{m-1})^T$ and
	write \eqref{general_algebraic_err} as
	\begin{equation}\label{general_matrix_form_error}
		\bm{B}^n\bm{E}_n=-\sum_{l=0}^{n-1}\bm{B}^{(n,l)}\bm{E}_l
		+\bm{b}^n+\tilde{\bm{\delta}}_n.
	\end{equation}
	Using $(\bm{B}^n)^{-1}\left[\eqref{general_matrix_form_error}_n
	-\eqref{general_matrix_form_error}_{n-1}\right]
	+(\bm{B}^n)^{-1}\bm{Q}^{(n,n-1)}(\bm{B}^{n-1})^{-1}
	\left[\eqref{general_matrix_form_error}_{n-1}
	-\eqref{general_matrix_form_error}_{n-2}\right]$ and applying Lemmas
	\ref{lemma3}~and \ref{lemmaLodd}, we find that
	\begin{equation}\label{sup-odd}
		\bm{E}_n=[\bm{I}_m+h\mathbb{V}_n^0]\bm{E}_{n-2}
		+\sum_{l=0}^{n-3}h^2{\mathbb{V}_n^l}\bm{E}_l+\hat{\bm{b}}_n
		+\overline{\bm{\delta}}_n,
	\end{equation}
	where
	\begin{align*}
		\hat{\bm{b}}_n&:=(\bm{B}^n)^{-1}(\bm{b}^{n}-\bm{b}^{n-1})
		+(\bm{B}^n)^{-1}\bm{Q}^{(n,n-1)}(\bm{B}^{n-1})^{-1}(\bm{b}^{n-1}
		-\bm{b}^{n-2}),\\
		\overline{\bm{\delta}}_n&:=(\bm{B}^n)^{-1}\tilde{\bm{\delta}}_n
		+(\bm{B}^n)^{-1}\bm{Q}^{(n,n-1)}
		(\bm{B}^{n-1})^{-1}\tilde{\bm{\delta}}_{n-1}.
	\end{align*}
	Consider Lemma~\ref{lemma7} with~$p=m+2$. For $i=0$, \dots, $m-1$,
	\begin{equation}\label{b_i^n}
		b_i^n=-\sum_{j=m}^{m+1}\gamma_jh^j\biggl(
		O(h^2)+\beta^n_{i,j}y^{(j)}(t_{n+1/2})
		+\sum_{l =0}^{n-1}\beta^{(n,l)}_{i,j}y^{(j)}(t_{l+1/2})\biggr)+O(h^{m+2}),
	\end{equation}
	and by Lemmas \ref{lemma1}~and \ref{lemma8}, and Taylor expansion,
	\begin{align*}
		b^n_i-b^{n-1}_{i} &= -\beta_{i,m}^n\gamma_{m}h^my^{(m)}(t_{n+1/2})
		-\gamma_{m}h^m\sum_{l=0}^{n-1}\beta_{i,m}^{(n,l)}y^{(m)}(t_{l+1/2}) \\
		&\qquad{}+\beta_{i,m}^{n-1}\gamma_{m}h^my^{(m)}(t_{n-1/2})
		+\gamma_{m}h^m\sum_{l=0}^{n-2}\beta_{i,m}^{(n-1,l)}y^{(m)}(t_{l+1/2})
		+O(h^{m+2}) \\
		&=-(\beta_{i,m}^n-\beta_{i,m}^{n-1})\gamma_{m}h^my^{(m)}(t_{n})
		-\gamma_{m}h^m\sum_{l=0}^{n-2}(\beta_{i,m}^{(n,l)}
		-\beta_{i,m}^{(n-1,l)})y^{(m)}(t_{l}) \\
		&\qquad{}-\beta_{i,m}^{(n,n-1)}\gamma_{m}h^my^{(m)}(t_{n-1/2})
		-\frac{1}{2}\gamma_mh^{m+1}(\beta_{i,m}^n+\beta_{i,m}^{n-1})y^{(m+1)}(t_n)\\
		&\qquad{}-\frac{1}{2}\gamma_mh^{m+1}\sum_{l=0}^{n-2}(
		\beta_{i,m}^{(n,l)}-\beta_{i,m}^{(n-1,l)})y^{(m+1)}(t_l)+O(h^{m+2}).
	\end{align*}
	Again by Lemmas \ref{lemma1}~and \ref{lemma8}, we know that for~$m>1$,
	\begin{equation}\label{eq:bn-bn-1}
		\begin{aligned}
			\bm{b}^{n}-\bm{b}^{n-1}&=-(\beta_{m-1,m}^n-\beta_{m-1,m}^{n-1})
			\gamma_{m}h^my^{(m)}(t_{n})\bm{e}_{m-1} \\
			&\qquad{}-\frac{1}{2}\gamma_mh^{m+1}(
			\beta_{m-1,m}^n+\beta_{m-1,m}^{n-1})y^{(m+1)}(t_n)\bm{e}_{m-1}
			+\bm{O}(h^{m+2})\\
			&=\bigl[k'(t_n,t_n)y^{(m)}(t_{n})
			+k(t_n,t_n)y^{(m+1)}(t_{n})\bigr]\alpha_m\gamma_mh^{m+1}\bm{e}_{m-1}
			+\bm{O}(h^{m+2}).
		\end{aligned}
	\end{equation}
	When~$m=1$, the case is a little more complicated. By Lemmas \ref{lemma8}~and
	\ref{lemma1},
	\begin{align*}
		b^n_0-b^{n-1}_{0}&=-(\beta_{0,1}^n-\beta_{0,1}^{n-1})\gamma_{1}hy'(t_{n})
		-\gamma_{1}h\sum_{l=0}^{n-2}
		(\beta_{0,1}^{(n,l)}-\beta_{0,1}^{(n-1,l)})y'(t_{l})\\
		&\qquad{}-\beta_{0,1}^{(n,n-1)}\gamma_{1}hy'(t_{n-1/2})
		-\frac{\gamma_1h^2}{2}(\beta_{0,1}^n+\beta_{0,1}^{n-1})y^{(2)}(t_n)\\
		&\qquad{}-\frac{\gamma_1h^2}{2}\sum_{l=0}^{n-2}
		(\beta_{0,1}^{(n,l)}-\beta_{0,1}^{(n-1,l)})y^{(2)}(t_l)
		-\frac{\gamma_{1}h^3}{4}\sum_{l=0}^{n-2}
		(\beta_{0,1}^{(n,l)}-\beta_{0,1}^{(n-1,l)})y^{(3)}(t_{l})+O(h^{3}) \\
		&=\bigl[k'(t_n,t_n)y'(t_n)+k(t_n,t_n)y''(t_n)\bigr]\alpha_1\gamma_1h^2\\
		&\qquad{}+\biggl[-\tfrac{\partial_{2}k(t_n,t_{n-1})}{6}\,y'(t_n)
		-h\sum_{l=0}^{n-2}\tfrac{\partial_1\partial_2k(t_n,t_l)}{6}\,y'(t_l)
		\biggr]h^2\gamma_1+O(h^3).
	\end{align*}
	By \eqref{e:BQ}, we know that $(\bm{B}^n)^{-1}\bm{Q}^{(n,n-1)} = \bm{I}_m - \bm{M}^{-1}\bm{N} + \bm{O}(h)$; further
	by Lemma \ref{lemma3}, when $m>1$,
	\begin{align*}
		\hat{\bm{b}}_n&=\tfrac{1}{k(t_n,t_n)}\,\bm{M}^{-1}(\bm{b}^{n}-\bm{b}^{n-1})
		+\tfrac{1}{k(t_{n-1},t_{n-1})}\bigl(\bm{I}_m-\bm{M}^{-1}\bm{N}\bigr)
		\bm{M}^{-1}(\bm{b}^{n-1}-\bm {b}^{n-2})+\bm{O}(h^{m+2})\\
		&=\alpha_m\gamma_mh^{m+1}\Bigl(\bigl[k'_{n,n}y^{(m)}(t_{n})
		+y^{(m+1)}(t_{n})\bigr]\bm{w}
		+\bigl[k'_{n-1,n-1}y^{(m)}(t_{n})
		+y^{(m+1)}(t_{n})\bigr](\bm{w}-w_0\bm{v})\Bigr)\\
		&\qquad{}+\bm{O}(h^{m+2})\\
		&=\mu_m h^{m+1}\bigl[k'_{n-1,n-1}y^{(m)}(t_{n})
		+y^{(m+1)}(t_{n})\bigr]\bm{q}+\bm{O}(h^{m+2}),
	\end{align*}
	and when $m=1$, by noticing that for this case $\bm{q} = q_0 = 0$,
	\begin{align*}
		\hat{\bm{b}}_n&=\tfrac{1}{k(t_n,t_n)}\bm{M}^{-1}(\bm{b}^{n}-\bm{b}^{n-1})
		+\tfrac{1}{k(t_{n-1},t_{n-1})}\,(\bm{I}_m-\bm{M}^{-1}\bm{N})
		\bm{M}^{-1}(\bm{b}^{n-1}-\bm {b}^{n-2})+\bm{O}(h^{m+2})\\
		&=\biggl[-\tfrac{\partial_2k(t_n,t_{n-1})}{6k(t_n,t_n)}y'(t_n)
		-h\sum_{l=0}^{n-2}
		\tfrac{\partial_1\partial_2k(t_n,t_l)}{6k(t_n,t_n)}y'(t_l)\biggr]
		h^2\gamma_1\bm{v} \\
		&\qquad{}+\biggl[
		-\tfrac{\partial_2k(t_{n-1},t_{n-2})}{6k(t_{n-1},t_{n-1})}\,y'(t_{n-1})
		-h\sum_{l=0}^{n-3}
		\tfrac{\partial_1\partial_2k(t_{n-1},t_l)}{6k(t_{n-1},t_{n-1})}\,y'(t_l)
		\biggr]h^2\gamma_1(\bm{v}-2\bm{v})+\bm{O}(h^{3}) \\
		&=-\tfrac{1}{6}\biggl[\tfrac{\partial_2k(t_n,t_{n-1})}{k(t_n,t_n)}
		-\tfrac{\partial_2k(t_{n-1},t_{n-2})}{k(t_{n-1},t_{n-1})}
		\biggr]y'(t_n)h^2\gamma_1\bm{v}
		-\tfrac{\partial_1\partial_2k(t_n,t_{n-2})}{6k(t_n,t_n)}
		y'(t_{n-2})h^3\gamma_1\bm{v} \\
		&\qquad{}-h\sum_{l=0}^{n-3}\tfrac{1}{6}\biggl[
		\tfrac{\partial_1\partial_2k(t_{n},t_l)}{k(t_{n},t_{n})}
		-\tfrac{\partial_1\partial_2k(t_{n-1},t_l)}{k(t_{n-1},t_{n-1})}\biggr]
		y'(t_l)h^2\gamma_1\bm{v}+\bm{O}(h^{3}) \\
		&=-\tfrac{\partial_1\partial_2k(t_n,t_{n-2})}{3k(t_n,t_n)}
		y'(t_n)h^3\gamma_1\bm{v}
		-h\sum_{l=0}^{n-3}\tfrac{\partial_1^2\partial_2k(\cdot,t_l)}{6k(t_n,t_n)}
		y'(t_l)h^3\gamma_1\bm{v}+\bm{O}(h^{3}) \\
		&=\mu_1 h^{2}\bigl[k'_{n-1,n-1}y'(t_{n})+y''(t_{n})\bigr]\bm{q}+\bm{O}(h^{3}),
	\end{align*}
	with bounded coefficient $\mu_m:=2\alpha_{m}\gamma_mw_0$.
	
	Brunner et al.~\cite[Section 3.3]{DG2009} showed that if
	$K\equiv1$~and $y^{(m)}(0)=0$, then the error in the DG solution is
	$\bm{E}_n=\frac{1}{2}\mu_mh^my^{(m)}(t_{n+1/2})\bm{q}+O(h^{m+1})$, so we
	now seek a solution of~\eqref{sup-odd} in the form
	\[
	\bm{E}_n =\tfrac{1}{2}\mu_mh^my^{(m)}(t_{n+1})\bm{q}+\hat{\bm{E}}_n,
	\]
	with the aim of proving $\hat{\bm{E}}_n=\bm{O}(h^{\nu_3-1})$, where $\nu_3:=\min\{m+2, m_1-1\}$. Substituting the above ansatz into \eqref{sup-odd} gives
	\begin{align*}
		\hat{\bm{E}}_n&-(\bm{I}_m+h\mathbb{V}_n^0)\hat{\bm{E}}_{n-2}
		-\sum_{l=0}^{n-3}h^2\mathbb{V}_n^l\hat{\bm{E}}_l \\
		&=-\mu_mh^{m}\Bigl(-hk'_{n-1,n-1}y^{(m)}(t_{n})
		+\tfrac{1}{2}y^{(m)}(t_{n+1})-\tfrac{1}{2}y^{(m)}(t_{n-1})
		-hy^{(m+1)}(t_n)\Bigr)\bm{q}\\
		&\qquad{}+\frac{\mu_m}{2}\,y^{(m)}(t_{n-1})h^{m+1}\mathbb{V}_n^0\bm{q}
		+\frac{\mu_m}{2}\sum_{l=0}^{n-3}
		y^{(m)}(t_{l+1})h^{m+2}\mathbb{V}_n^l\bm{q}+\bm{O}(h^{\nu_3}).
	\end{align*}
	By Lemma~\ref{lemmaLodd}, we use $\mathbb{V}_n^l\bm{q}=\bm{O}(h)$ for each
	$l\geq 1$, and $\mathbb{V}_n^0\bm{q}=-2k'_{n-1,n-1}\bm{q}+\bm{O}(h)$. Taylor
	expansoin of the first right-hand side term gives
	\[
	\hat{\bm{E}}_n-[\bm{I}_m+h\mathbb{V}_n^0]\hat{\bm{E}}_{n-2}
	-\sum_{l=0}^{n-3}h^2\mathbb{V}_n^l\hat{\bm{E}}_l =\bm{O}(h^{\nu_3}),
	\]
	and so taking norms results in
	\[
	\|\hat{\bm{E}}_n\|\leq (1+Ch)\|\hat{\bm{E}}_{n-2}\|
	+ Ch^2\sum_{l=0}^{n-2}\|\hat{\bm{E}}_l\|+Ch^{\nu_3}.
	\]
	Note that from \eqref{general_matrix_form_error}~and \eqref{b_i^n}
	with~$n=0$, and by Taylor expansion about~$t_0$,
	\begin{align*}
		\hat{\bm{E}}_0&=\bm{E}_0-\tfrac12\mu_mh^my^{(m)}(0)\bm{q}+\bm{O}(h^{m+1})
		=(\bm{B}^0)^{-1}\bm{b}^0 +(\bm{B}^0)^{-1}\tilde{\bm{\delta}}_0
		+\bm{O}(h^{m+1})\\
		&=-\alpha_{m}\gamma_mh^my^{(m)}(t_{1/2})\bm{w}+\bm{O}(h^{\min\{m+1,m_1-1\}})
		=\bm{O}(h^{\min\{m+1, m_1-1\}}),
	\end{align*}
	since $y^{m}(0)=0$. By an induction argument and Gronwall inequalities, the
	estimate $\|\hat{\bm{E}}_n\| = O(h^{\nu_3-1})$ follows. We have thus shown that
	\begin{align*}
		\bm{E}_n =\tfrac12\mu_{m}h^{m}y^{(m)}(t_{n+1})\bm{q}+\bm{O}(h^{\nu_3-1}),
	\end{align*}
	with $\bm{q}=w_0^{-1}\bm{w}-\frac12\bm{v}=(0,3,0,7,\dots,2m-3,0)^T$,
	so from \eqref{general-local-err} with~$p=m+2$ and Lemma~\ref{lemma6},
	\begin{align*}
		e(t_n+sh)&=\sum_{j=0}^{m-1}P_j(s)E_j^n
		+\gamma_my^{(m)}(t_{n+1/2})h^mP_m(s)+O(h^{m+1}) \\
		&=\frac{\mu_{m}}{2}\,h^{m}y^{(m)}(t_{n+1})\sum_{j=0}^{m-1}P_j(s)q_j
		+\gamma_my^{(m)}(t_{n+1})h^mP_m(s)+O(h^{\nu_3-1}).
	\end{align*}
	Recalling the identity $\sum_{k=0}^{(m-1)/2}(4k+3)P_{2k+1}(s)
	=P_{m+1}'(s)$~\cite[Proposition A.2]{DG2009}, it follows that
	\begin{align*}
		\sum_{j=0}^{m-1}P_j(s)q_j&=\sum_{k=0}^{(m-3)/2}(4k+3)P_{2k+1}(s)
		=\sum_{k=0}^{(m-1)/2}(4k+3)P_{2k+1}(s)-(2m+1)P_{m}(s)\\
		&=P_{m+1}'(s)-(2m+1)P_{m}(s),
	\end{align*}
	and by noticing that $\gamma_m-\frac{1}{2}\mu_{m}(2m+1) = 0$, we have
	\begin{equation}\label{odd-e-pre}
		\begin{aligned}
			e(t_n+sh)&=\tfrac{1}{2}\mu_{m}h^{m}y^{(m)}(t_{n+1})
			\bigl[P_{m+1}'(s)-(2m+1)P_{m}(s)\bigr]
			+\gamma_my^{(m)}(t_{n+1})h^mP_m(s)+O(h^{\nu_3-1}) \\
			&=\tfrac{1}{2}\mu_{m}h^{m}y^{(m)}(t_{n+1})P_{m+1}'(s)
			+\bigl[\gamma_m-\tfrac{1}{2}\mu_{m}(2m+1)\bigr]h^m y^{(m)}(t_{n+1})P_m(s)
			+O(h^{\nu_3-1}) \\
			&=\tfrac{1}{2}\mu_{m}h^{m}y^{(m)}(t_{n+1})P_{m+1}'(s)+O(h^{\nu_3-1}).
		\end{aligned}
	\end{equation}
	Hence, we have the superconvergence at the zeros of~$P_{m+1}'(s)$.
	\medskip
	
	\noindent{\textbf{Case II: \bm{$m$} is even.}}
	Let $p=m+1$ and use $(\bm{B}^n)^{-1}[\eqref{general_matrix_form_error}_n
	-\bm{F}\eqref{general_matrix_form_error}_{n-1}]$ to deduce that
	\begin{equation}\label{eq:general even}
		\bm{E}_n=(\bm{B}^n)^{-1}\tilde{\bm{Q}}^n\bm{E}_{n-1}
		+\sum_{l=0}^{n-2}(\bm{B}^n)^{-1}\tilde{\bm{Q}}^{(n,l)}\bm{E}_{l}
		+\tilde{\bm{b}}^n+\tilde{\tilde{\bm{\delta}}}_n.
	\end{equation}
	Here, $\tilde{\tilde{\bm{\delta}}}_n:=(\bm{B}^n)^{-1}\bigl(
	\tilde{\bm{\delta}}_n-\bm{F}\tilde{\bm{\delta}}_{n-1}\bigr)$ and
	$\tilde{\bm{b}}^n := (\bm{B}^n)^{-1}\bigl(\bm{b}^n-\bm{F}\bm{b}^{n-1}\bigr)$,
	with $\tilde{\bm{Q}}^n$~and $\tilde{\bm{Q}}^{(n,l)}$ defined
	in~\eqref{e:Q}.  By \eqref{invBn}, \eqref{eq:bn-bn-1} and Lemmas
	\ref{lemma1}, \ref{lemma7} and \ref{lemma3}, we have
	\begin{align*}
		\tilde{\bm{b}}^n&=(\bm{B}^n)^{-1}(\bm{b}^n-\bm{F}\bm{b}^{n-1})
		=(\bm{B}^n)^{-1}\bigl[(\bm{b}^n-\bm{b}^{n-1})
		-\bm{N}\bm{M}^{-1}\bm{b}^{n-1}\bigr]
		=-(\bm{B}^n)^{-1}\bm{N}\bm{M}^{-1}\bm{b}^{n-1}+\bm{O}(h^{m+1})\\
		&=-\tfrac{1}{k(t_n,t_n)}\bm{M}^{-1}\bm{N}\bm{M}^{-1}\bigl(
		\gamma_mh^m\beta_{m-1,m}^{n-1}y^{(m)}(t_{n-1/2})
		+\sum_{l=0}^{n-2}\gamma_mh^m\beta_{m-1,m}^{(n-1,l)}
		y^{(m)}(t_{l+1/2})\biggr)\bm{e}_{m-1}\\
		&\qquad{}+\bm{O}(h^{m+1})\\
		&=-\tfrac{1}{k(t_n,t_n)}\biggl(\gamma_mh^m\beta_{m-1,m}^{n-1}y^{(m)}(t_{n-1/2})
		+\sum_{l=0}^{n-2}\gamma_mh^m\beta_{m-1,m}^{(n-1,l)}y^{(m)}(t_{l+1/2})
		\biggr)\bm{M}^{-1}\bm{N}\bm{w}+\bm{O}(h^{m+1}) \\
		&=-\tfrac{1}{k(t_n,t_n)}\biggl(\gamma_mh^m\beta_{m-1,m}^{n-1}y^{(m)}(t_{n-1/2})
		+\sum_{l=0}^{n-2}\gamma_mh^m\beta_{m-1,m}^{(n-1,l)}y^{(m)}(t_{l+1/2})
		\biggr)w_0\bm{v}+\bm{O}(h^{m+1}).
	\end{align*}
	When~$m>1$, by Lemma~\ref{lemma1},
	\[
	\tilde{\bm{b}}^n=h^m\alpha_{m}\gamma_my^{(m)}(t_{n-1/2})w_0\bm{v}+O(h^{m+1}).
	\]
	When~$m=1$, the case is a little more complicated. By Lemma~\ref{lemma1},
	\[
	\tilde{\bm{b}}^n=h\gamma_1 \biggl[\alpha_{1}y'(t_{n-1/2})
	+\frac{h}{6}\sum_{l=0}^{n-2}\partial_2k(t_{n-1},t_l)y'(t_{l+1/2})
	\biggr]w_0\bm{v}+O(h^{2}).
	\]
	Therefore, we can write $\tilde{\bm{b}}^n=\hat{a}_nh^m\bm{v} + O(h^{m+1})$
	where the scalars~$\hat{a}_n$ are bounded as~$h\to0$.
	
	By Lemma~\ref{lemmaLeven}, Equation~\eqref{eq:general even} becomes
	\begin{equation}\label{sup_err_eq}
		\begin{aligned}
			\bm{E}_n&=(\bm{I}_m+h\mathbb{W}_n)\bm{E}_{n-1}+\hat{a}_nh^m\bm{v}
			-\frac{h^2}{k(t_n,t_n)}\sum_{l=0}^{n-2}\partial_1^2k(t_{n-1/2},t_l)
			\bm{N}\bm{E}_l
			+ \tilde{\tilde{\bm{\delta}}}_n+\bm{O}(h^{m+1}) \\
			&=(1-hk_{n,n}')\bm{E}_{n-1} +
			\hat{d}_{n,1}h^{\nu_2}\bm{v}+h\bar{a}_n^1E_{n-1}^0\bm{e}_0
			-\frac{h^2}{k(t_n,t_n)}\sum_{l=0}^{n-2}
			\partial_1^2k(t_{n-1/2},t_l)E_l^0\bm{e}_0
			+\tilde{\tilde{\bm{\delta}}}_n+\bm{O}(h^{\nu_2+1}),
		\end{aligned}
	\end{equation}
	where we have used the global convergence result in Theorem~\ref{perturbed}
	that $\|\bm{E}_n\|_\infty=\bm{O}(h^{\nu_2-1})$, and the scalars~$\hat{d}_{n,1}$
	are bounded as~$h\to0$. We now split $\bm{E}_n$ into orthogonal pieces,
	writing it as $\bm{E}_n=E_n^0\bm{e}_0+\hat{\bm{E}}_n$, where $\hat{E}_n^0=0$.
	Noting that $v_0=0$ when $m$ is even, we seee that \eqref{sup_err_eq}
	decouples into the two equations,
	\begin{equation}\label{decouple1}
		E_n^0 =(1+O(h))E_{n-1}^0 + \sum_{l=0}^{n-2}O(h^2)E_l^0 +O(h^{\nu_4}),
	\end{equation}
	and
	\begin{equation}\label{decouple2}
		\begin{aligned}
			\hat{\bm{E}}_n&= (1-hk_{n,n}')\hat{\bm{E}}_{n-1}+\hat{d}_{n,1}h^{\nu_2}\bm{v}
			+\tilde{\tilde{\bm{\delta}}}_n+\bm{O}(h^{\nu_2+1})\\
			&=(1-hk_{n,n}')\hat{\bm{E}}_{n-1}+\hat{d}_{n,1}h^{\nu_2}\bm{v}
			+\bm{O}(h^{\nu_4}),
		\end{aligned}
	\end{equation}
	with~$\nu_4:=\min\{m+1, m_1-1\}$. Using the triangle inequality on the scalar
	equation~\eqref{decouple1} gives
	\[
	|E_n^0| \leq (1+Ch) |E_{n-1}^0| + Ch^2 \sum_{l=0}^{n-2}|E_l^0| + Ch^{\nu_4}.
	\]
	Similarly to the induction in the previous proof, we have
	$|E_n^0|\leq Ch^{\nu_4-1}$, and a simple inductive argument on the vector
	equation~\eqref{decouple2} gives
	\[
	\hat{\bm{E}}_n=\bar{d}_{n,0}\hat{\bm{E}}_0+\bar{d}_{n,1}h^{\nu_2-1}\bm{v}
	+\bm{O}(h^{\nu_4-1}),
	\]
	where $\bar{d}_{n,0}:=\prod_{p=1}^{n}(1-hk_{p,p}')$~and
	$\bar{d}_{n,1}:=h\sum_{p=1}^{n}[\prod_{q=p+1}^n(1-hk_{q,q}')]\hat{d}_{p,1}$ are
	bounded as~$h\to0$.
	
	From \eqref{general_algebraic_err} with~$n=0$, and Lemma~\ref{lemma7}
	with~$p=m+1$, we can obtain
	\begin{equation}\label{e:E0-even}
		\bm{E}_0=O(h^{\min\{m, m_1-1\}}).
	\end{equation}
	Therefore,
	\[
	\bm{E}_n=\bar{d}_{n,1}h^{\nu_2-1}\bm{v}+\bm{O}(h^{\nu_4-1}),
	\]
	so from~\eqref{general-local-err} with~$p=m+1$, and Lemma~\ref{lemma6},
	\begin{equation}\label{even_sup_e}
		e(t_n+sh)=\sum_{j=0}^{m-1}P_j(s)E_j^n+O(h^{m})
		=\bar{d}_{n,1}h^{\nu_2-1}\sum_{j=0}^{m-1}v_jP_j(s)+O(h^{\nu_4-1}).
	\end{equation}
	Since~$m$ is even~\cite[Proposition A.2]{DG2009},
	\begin{equation}\label{e:vp}
		\sum_{j=0}^{m-1}v_jP_j(s)=-2\sum_{k=0}^{(m-2)/2}(4k+3)P_{2k+1}(s)=-2P_m'(s),
	\end{equation}
	and so
	\[
	e(t_n+sh)=-2\bar{d}_{n,1}h^{\nu_2-1}P_m'(s)+O(h^{\nu_4-1}),
	\]
	which implies that $e(t_n+sh)$ has order of accuracy~$h^{\nu_4-1}$ at the
	$m-1$~roots of~$P_m'(s)=0$, i.e., at the interior Lobatto points.
\end{proof}

%---------------------------------------------------------------
\section{Convergence analysis for index-2 IAEs} \label{sec:4}
\subsection{The global convergence of DG methods for index-2 IAEs}\label{subsec:4.1}

The convergence analysis of the index-$2$ IAE~\eqref{IAE} benefits
significantly from the convergence theory established for the perturbed DG
method for first-kind VIEs in Theorem~\ref{perturbed}.

\begin{theorem}\label{thr4}
	Assume that the hypotheses of Theorem~\ref{thr1} are satisfied for~$d=m$, so
	that the exact solution of the IAE~\eqref{IAE} satisfies
	$x_1$, $x_2\in C^{m+2}(I)$.  Then, for all uniform meshes~$I_h$ with $h<\bar h$,
	the DG solution $x_{1,h}$, $x_{2,h}\in S_{m-1}^{(-1)}(I_h)$ converges
	to the exact solution uniformly on~$I$ with the attainable global order of
	convergence described by
	\[
	\| x_1-x_{1,h} \|_\infty:=\sup_{t\in I}|x_1(t)-x_{1,h}(t)|
	\leq C\times\begin{cases}
		h^{m},&\text{if $m$ is odd}, \\
		h^{m-1},&\text{if $m$ is even},
	\end{cases}
	\]
	and
	\[
	\|x_2-x_{2,h}\|_\infty:=\sup_{t\in I}|x_2(t)-x_{2,h}(t)|
	\leq C\times\begin{cases}
		h^{m-2},&\text{if $m$ is odd}, \\
		h^{m-3},&\text{if $m$  is even}.
	\end{cases}
	\]
\end{theorem}
\begin{proof}
	Recall that $x_1$~and $x_2$ satisfy the equations
	\begin{equation}\label{eq1}
		\int_{t_n}^{t_{n+1}}x_{1}(s)\phi(s)\,\ud s
		+\int_{t_n}^{t_{n+1}}\biggl[\int_0^s\biggl(
		K_{11}(s,\tau)x_1(\tau)+K_{12}(s,\tau)x_2(\tau)\biggr)\,\ud\tau
		\biggr]\phi(s)\,\ud s
		=\int_{t_n}^{t_{n+1}}f_1(s)\phi(s)\,\ud s
	\end{equation}
	and
	\begin{equation}\label{eq1'}
		\int_{t_n}^{t_{n+1}}\int_0^s\biggl(K_{21}(s,\tau)x_1(\tau)\,\ud\tau
		\biggr)\phi(s)\,\ud s=\int_{t_n}^{t_{n+1}}f_2(s)\phi(s)\,\ud s.
	\end{equation}
	Let $e_{p,h}:=x_p-x_{p,h}$ for $p=1$, $2$. Since the second equation
	of~\eqref{IAE} is a VIE of the first kind, by Theorem \ref{perturbed}
	and Brunner et al.~\cite[Theorems 4.3 and 4.5]{DG2009}, we have
	\begin{equation}
		\| e_{1,h} \|_\infty=\|x_{1}-x_{1,h}\|_\infty
		\leq C\times\begin{cases}
			h^{m},&\text{if $m$ is odd}, \\
			h^{m-1},&\text{if $m$ is even}.
		\end{cases}
	\end{equation}
	From \eqref{eq1}~and \eqref{DGeq1}, we obtain the following system of error
	equations:
	\begin{equation}\label{e2h}
		\int_{t_n}^{t_{n+1}}\biggl(\int_0^s K_{12}(s,\tau)e_{2,h}(\tau)\,ud\tau
		\biggr)\phi(s)\,\ud s
		=-\int_{t_n}^{t_{n+1}}\biggl[e_{1,h}(s)+\biggl(
		\int_0^s K_{11}(s,\tau)e_{1,h}(\tau)\,\ud\tau
		\biggr)\biggr]\phi(s)\,\ud s.
	\end{equation}
	The structure of the above equation is as same as~\eqref{yh}, so again by
	Theorem~\ref{perturbed}, we have
	\[
	\|e_{2,h}\|_\infty\leq C\times\begin{cases}
		h^{m-2},&\text{if $m$ is odd}, \\
		h^{m-3},&\text{if $m$ is even}.
	\end{cases}
	\]
	The proof is complete.
\end{proof}
%-----------------------------------------------------------------------
\begin{remark}\label{rmk1}
	Liang and Brunner~\cite{2016siam} considered the discontinuous
	piecewise polynomial collocation (DC) method for the index-$2$ IAE.
	For our DG convergence results, odd~$m$ case corresponds to the
	case~$\rho_m = -1$, and the even~$m$ case corresponds to the case
	$\rho_m =1$. However, for the DC method, there is no superconvergence analysis
	for the index-$2$ IAE~\eqref{IAE}, whereas in the following we will show the
	superconvergence property of the DG methods.
\end{remark}
%-----------------------------------------------------------------------

%-----------------------------------------------------------------------
\subsection{The superconvergence analysis of the DG method for index-2 IAEs}\label{subsec:4.2}
%-----------------------------------------------------------------------

By Theorem~\ref{sup_perturbed} (see also Brunner et
al.~\cite[Theorems 4.3 and 4.5]{DG2009}), we have the following
superconvergence result for the DG solution without perturbation of the
first-kind VIE~\eqref{V1}.
%-----------------------------------------------------------------------
\begin{theorem}\label{thr5}
	Let $\tilde{y}_h\in S_{m-1}^{(-1)}(I_h)$ be the DG solution for~\eqref{V1}.
	\begin{enumerate}
		\item Assume that $g\in C^{d+4}(I)$, $k\in C^{d+4}(D)$, if $m$ is odd,
		and the exact solution $y$ satisfies $y^{(m)}(0)=0$, then for each $n$,
		\begin{align*}
			|y(t_n+s_rh)-\tilde{y}_h(t_n+s_rh)|\leq Ch^{m+1},
		\end{align*}
		where the $\{s_r\}_{r=0}^{m-1}$ satisfy $P_{m+1}'(s_r)=0$.
		\item Assume that $g\in C^{d+3}(I)$, $k\in C^{d+3}(D)$, if $m$ is even, then
		for each $n$,
		\begin{align*}
			|y(t_n+s_rh)-\tilde{y}_h(t_n+s_rh)|\leq Ch^{m},
		\end{align*}
		where the $\{s_r\}_{r=1}^{m-1}$ satisfy $P_{m}'(s_r)=0$.
	\end{enumerate}
\end{theorem}
%-----------------------------------------------------------------------

Now we will show that the local superconvergence of the index-2 IAE~\eqref{IAE}
is inherited from the above superconvergence result for first-kind VIEs, and
that the superconvergence result of the component~$x_2$ is global for~$m$ odd.

%-----------------------------------------------------------------------
\begin{theorem}\label{thr6}
	Assume that the conditions of Theorem~\ref{thr4} are satisfied when~$m$ is
	even, and if $m$ is odd, the conditions of Theorem~\ref{thr1} hold such that
	exact solutions $x_1$, $x_2 \in C^{m+3}(I)$, with $x_1$ satisfying
	$x_1^{(m)}(0)=0$. Then, we have the following superconvergence results for the
	DG solution:
	\begin{align*}
		|e_{1,h}(t_n+s_rh)| & \leq C\times\begin{cases}
			h^{m+1},&\text{if $m$ is odd}, \\
			h^{m},&\text{if $m$ is even}, \\
		\end{cases}\\
		\| e_{2,h} \|_\infty = \sup_{t\in I} |e_{2,h}(t)|
		&\leq Ch^{m-1},\quad\text{if $m$ is odd}, \\
		|e_{2,h}(t_n+s_rh)| &\leq Ch^{m-2},\quad\text{if $m$ is even}.
	\end{align*}
	Here, if $m$ is odd, $\{s_r\}_{r=0}^{m-1}$ satisfy $P_{m+1}'(s_r)=0$,
	whereas if $m$ is even, $\{s_r\}_{r=1}^{m-1}$ satisfy $P_{m}'(s_r)=0$.
\end{theorem}
%-----------------------------------------------------------------------
%-----------------------------------------------------------------------
\begin{proof}
	The results for~$\vert e_{1,h}(t_n+s_rh) \vert$ are obtained directly from
	Theorem~\ref{thr5}. For the analysis of~$e_{2,h}$, we treat two
	cases separately.
	\bigskip
	
	\noindent\textbf{Case I: $m$ is odd}.
	Note that the perturbed term~$\delta$ in~\eqref{yh} now has the specific
	form
	\[
	\delta(s)=-e_{1,h}(s)-\int_{0}^{s}K_{11}(s, \tau)e_{1,h}(\tau)\,\ud\tau,
	\]
	so
	\begin{equation}\label{specific-delta}
		\delta(t_n+sh)=-e_{1,h}(t_n+sh)
		-\int_{0}^{t_n+sh}K_{11}(t_n+sh,\tau)e_{1,h}(\tau)\,\ud\tau.
	\end{equation}
	Therefore, by Taylor expansion and Theorem~\ref{thr4}, we have
	\begin{align*}
		\delta(t_n+sh)-\delta(t_{n-1}+sh)&=-[e_{1,h}(t_n+sh)-e_{1,h}(t_{n-1}+sh)]\\
		&\qquad{}-\int_{0}^{t_{n-1}+sh}\bigl[
		K_{11}(t_n+sh,\tau)-K_{11}(t_{n-1}+sh,\tau)
		\bigr]e_{1,h}(\tau)\,\ud\tau\\
		&\qquad{}-\int_{t_{n-1}+sh}^{t_n+sh}K_{11}(t_n+sh,\tau)e_{1,h}(\tau)\,\ud\tau\\
		&=-[e_{1,h}(t_n+sh)-e_{1,h}(t_{n-1}+sh)] + O(h) \|e_{1,h}\|_{\infty} \\
		&=-[e_{1,h}(t_n+sh)-e_{1,h}(t_{n-1}+sh)] + O(h^{m+1}).
	\end{align*}
	Since $e_{1,h}$ is the DG error of the first-kind VIE, and we have assumed
	$x_1^{(m)}(0)=0$ for~$m$ odd, by~\eqref{odd-e-pre} the error
	component~$e_{1,h}(t_n+sh)$ can be written as
	\begin{align*}
		e_{1,h}(t_n+sh)=\tfrac{1}{2}\mu_{m}h^{m}x_1^{(m+1)}(t_{n+1})P_{m+1}'(s)
		+O(h^{m+1}).
	\end{align*}
	Thus, by Taylor expansion,
	\[
	e_{1,h}(t_n+sh)-e_{1,h}(t_{n-1}+sh)
	=\tfrac{1}{2}\mu_{m}h^{m}\bigl[
	x_1^{(m+1)}(t_{n+1})-x_1^{(m+1)}(t_{n})\bigr]P_{m+1}'(s)
	+O(h^{m+1})=O(h^{m+1}),
	\]
	and so
	\[
	\delta(t_n+sh)-\delta(t_{n-1}+sh)=O(h^{m+1}).
	\]
	Returning to the proof of Theorem~\ref{sup_perturbed} in the case when~$m$ is
	odd, we have $\overline{\bm{\delta}}_n=O(h^{m})$. Then, from~\eqref{sup-odd}
	and the estimate for~$\hat{\bm{b}}_n$ in the proof of
	Theorem~\ref{sup_perturbed}, by an induction argument and Gronwall's
	inequality, we have $e_{2,h}(t_n+sh)=O(h^{m-1})$.
	\bigskip
	
	\noindent\textbf{Case II: $m$ is even.} From~\eqref{e:E0-even} without the
	perturbation, we have $\bm{E}_0=\bm{O}(h^m)$, so $e_{1,h}(t_0+sh)=O(h^m)$. Then
	from~\eqref{specific-delta} with~$n=0$, and Theorem~\ref{thr4},
	\[
	\delta(t_0+sh)=-\int_{0}^{sh}K_{11}(sh,\tau)\,\ud\tau\|e_{1,h}\|_\infty+O(h^m)
	=O(h)\|e_{1,h}\|_\infty+O(h^m)=O(h^m),
	\]
	which means $\tilde{\tilde{\bm{\delta}}}_0=O(h^{m-1})$. Moreover,
	from~\eqref{specific-delta},
	\begin{align*}
		\delta(t_n+sh)&=-e_{1,h}(t_n+sh)
		-\int_{0}^{t_n+sh}K_{11}(t_n+sh,\tau)e_{1,h}(\tau)\,\ud\tau \\
		&=-e_{1,h}(t_n+sh)
		-h\int_{0}^{s}K_{11}(t_n+sh,t_n+\tau h)e_{1,h}(t_n+\tau h)\,\ud\tau\\
		&\qquad{}-h\sum_{l=0}^{n-1}\int_0^1
		K_{11}(t_n+sh,t_l+\tau h)e_{1,h}(t_l+\tau h)\,\ud\tau.
	\end{align*} 
	From~\eqref{even_sup_e} without perturbation, and \eqref{e:vp}, we have
	\[
	e_{1,h}(t_n+sh)=\tilde{\bar{d}}_{n,0}h^{m-1}\sum_{j=0}^{m-1}v_jP_j(s)+O(h^{m})
	=-2\tilde{\bar{d}}_{n,0}h^{m-1}P_m'(s)+O(h^m),
	\]
	where the scalars~$\tilde{\bar{d}}_{n,0}$ are bounded as~$h\to0$.
	Further, since $P_m'(s)$ is a real polynomial of degree not exceeding~$m-1$, we
	can use Lagrange interpolation to get $P_m'(s)=\sum_{r=1}^{m-1}P_m'(s_r)L_r(s)$,
	so that
	\begin{align*}
		\delta(t_n+sh)&=2\tilde{\bar{d}}_{n,0}h^{m-1}\sum_{r=1}^{m-1}P_m'(s_r)L_r(s)
		+O(h^m)\\
		&\qquad{}+2\tilde{\bar{d}}_{l,0}h\sum_{l=0}^{n-1}\int_0^1
		K_{11}(t_n+sh,t_l+\tau h)\biggl[
		h^{m-1}\sum_{r=1}^{m-1}P_m'(s_r)L_r(\tau)\biggr]\,\ud\tau+O(h^m).
	\end{align*}
	Therefore, if $\{s_r\}_{r=1}^{m-1}$ satisfy $P_m'(s_r)=0$, then
	$\delta(t_n+sh)=O(h^m)$, and $\tilde{\tilde{\bm{\delta}}}_n=\bm{O}(h^{m-1})$.
	Now return to the proof of Theorem~\ref{sup_perturbed} when~$m$ is even and
	$m_1=m$, so that $\nu_4-1=\min\{m, m-2\}=m-2$ and, by Theorem~\ref{thr4},
	$\nu_2-1=\min\{m-1,m-3\}=m-3$.
	
	Again from~\eqref{even_sup_e} together with the above estimation,
	\[
	e_2(t_n+sh)=\tilde{\bar{d}}_{n,1}h^{m-3}\sum_{j=0}^{m-1}v_jP_j(s)+O(h^{m-2})
	=-2\tilde{\bar{d}}_{n,1}h^{m-3}P_m'(s)+O(h^{m-2}),
	\]
	where the scalars~$\tilde{\bar{d}}_{n,1}$ are bounded as~$h\to0$.
	Hence, if $\{s_r\}_{r=1}^{m-1}$ satisfy $P_m'(s_r)=0$, then
	$e_{2,h}(t_n+s_rh)=O(h^{m-2})$.
\end{proof}
%-------------------------------------------

%-------------------------------------------
\begin{remark}
	For the convergence rate of the $x_1$~component, we obtain one order higher
	at some special points than the global convergence order, which is also one
	order higher than the global convergence rate of the DC
	method~\cite[Theorem 4.5]{2016siam}. For the $x_2$~component, we also get
	one order higher than the global convergence order of the DG method: at
	some special points for~$m$ even, and globally for~$m$ odd.
\end{remark}
%-------------------------------------------

%-------------------------------------------
\begin{remark}
	For the index-$1$ IAE, the global superconvergence result of DG solution of
	the first component is usually obtained by
	iteration~\cite[Theorem 4.2]{gao2023}. However, for the index-$2$ IAE
	\eqref{IAE}, we obtain local and global superconvergence results directly,
	without using iteration.
\end{remark}
%%%%%%%%%%%%%%%%%%%%%%%%%%%%%%%%%%%%%%%%%%%%%%%%
%%%%%%%%%%%%%%%%%%%%%%%%%%%%%%%%%%%%%%%%%%%%%%%%%%%

\section{Numerical experiments}\label{sec:5}
We now present some representative examples to illustrate the previous convergence results.
\begin{example}\label{ex1}
	Consider the following index-$2$ IAE from Liang and Brunner~\cite{2016siam}
	with~$r=1$:
	\begin{align*}
		\left\{
		\begin{aligned}
			x_1(t)+\int_0^t[(t-s)x_1(s)+e^{t-s}x_2(s)]\,\ud s&=f_1(t),\\
			\int_0^t e^{2t-s}x_1(s)\,\d s&=f_2(t),
		\end{aligned}\right.
	\end{align*}
	on $I:=[0,1]$. We choose $f_1(t)$~and $f_2(t)$ such that the exact
	solution is $x_1(t)=t e^{-t}$~and $x_2(t)=\cos t$. Here
	$x_1^{(m)}(0)\not=0$ and $x_2^{(m)}(0)=0$ when~$m$ is odd.
\end{example}

Tables \ref{tab1}~and \ref{tab2} list the maximum absolute errors of the DG
solutions. The observed convergence results agree with Theorem~\ref{thr4}.
Tables \ref{tab3}~and \ref{tab4} list the maximum absolute errors of the
local DG solution at points~$t_n+s_rh$ with~$P'_{m+1}(s_r)=0$ when~$m$ is odd,
but $P'_{m}(s_r)=0$ when $m$ is even.  The numerical results show
that, due to $x_1^{(m)}(0)\neq0$, the superconvergence occurs only when~$m$ is
even, in agreement with Theorems \ref{thr6}~and \ref{thr4}.

\begin{table}[!t]
	\caption{The maximum absolute DG errors for $x_1$ component of Example \ref{ex1}\label{tab1}}%
	\begin{tabular*}{\columnwidth}{@{\extracolsep\fill}l>{\ttfamily}l>{\ttfamily}l>{\ttfamily}l>{\ttfamily}l@{\extracolsep\fill}}
		\hline
	$N$   & $m=3$    & $m=4$  & $m=5$ & $m=6$ \\
	\hline
	$2^2$ & 7.33E-04 & 4.06E-05 & 4.39E-07 & 4.65E-07 \\
	$2^3$ & 1.01E-04 & 4.96E-06 & 1.50E-08 & 1.39E-08 \\
	$2^4$ & 1.36E-05 & 6.09E-07 & 4.93E-10 & 4.27E-10 \\
	$2^5$ & 1.70E-06 & 7.56E-08 & 1.58E-11 & 1.32E-11 \\
	Order & 3.000  & 3.010 & 4.964 & 5.016  \\
	\hline
	\end{tabular*}
\end{table}

\begin{table}[!t]
	\caption{The maximum absolute DG errors for $x_2$ component of Example \ref{ex1}\label{tab2}}%
	\begin{tabular*}{\columnwidth}{@{\extracolsep\fill}l>{\ttfamily}l>{\ttfamily}l>{\ttfamily}l>{\ttfamily}l@{\extracolsep\fill}}
		\hline
		$N$  & $m=3$    & $m=4$  & $m=5$  & $m=6$ \\
		\hline
		$2^2$ & 1.11E-01 & 1.47E-02 & 1.56E-04 & 7.03E-05 \\
		$2^3$ & 5.83E-02 & 7.92E-03 & 2.75E-05 & 1.06E-05 \\
		$2^4$ & 3.09E-02 & 4.08E-03 & 2.88E-06 & 1.42E-06 \\
		$2^5$ & 1.50E-02 & 2.07E-03 & 3.64E-07 & 1.84E-07 \\
		Order & 1.042  & 0.979  & 2.984 & 2.948   \\
		\hline
	\end{tabular*}
\end{table}

\begin{table}[!t]
	\caption{The local DG errors at $t_n+s_rh$ for $x_1$ component of Example \ref{ex1}}\label{tab3}%
	\begin{tabular*}{\columnwidth}{@{\extracolsep\fill}l>{\ttfamily}l>{\ttfamily}l>{\ttfamily}l>{\ttfamily}l@{\extracolsep\fill}}
		\hline
		$N$  & $m=3$    & $m=4$  & $m=5$  & $m=6$ \\
		\hline
		$2^2$ & 8.12E-05 & 1.71E-06 & 2.61E-08 & 1.92E-08 \\
		$2^3$ & 1.04E-05 & 1.14E-07 & 8.45E-10 & 3.33E-10 \\
		$2^4$ & 1.31E-06 & 7.34E-09 & 2.67E-11 & 5.49E-12 \\
		$2^5$ & 1.64E-07 & 4.66E-10 & 8.38E-13 & 8.81E-14 \\
		Order & 3.000     & 3.977     & 4.994     & 5.962   \\
		\hline
	\end{tabular*}
\end{table}

\begin{table}[!t]
	\caption{The local DG errors at $t_n+s_rh$ for $x_2$ component of Example \ref{ex1}}\label{tab4}%
	\begin{tabular*}{\columnwidth}{@{\extracolsep\fill}l>{\ttfamily}l>{\ttfamily}l>{\ttfamily}l>{\ttfamily}l@{\extracolsep\fill}}
		\hline
		$N$  & $m=3$    & $m=4$  & $m=5$  & $m=6$ \\
		\hline
		$2^2$ & 2.66E-02 & 2.46E-04 & 2.11E-05 & 1.00E-06 \\
		$2^3$ & 1.43E-02 & 6.43E-05 & 2.79E-06 & 7.16E-08 \\
		$2^4$ & 7.35E-03 & 1.64E-05 & 3.56E-07 & 4.67E-09\\
		$2^5$ & 3.73E-03 & 4.12E-06 & 4.53E-08 & 2.98E-10 \\
		Order & 0.979 & 1.993  & 2.974  & 3.970 \\
		\hline
	\end{tabular*}
\end{table}

\begin{example}\label{ex2}
	Consider the index-$2$ IAE
	\begin{align*}
		\left\{
		\begin{aligned}
			x_1(t)+\int_0^t[(t-s)x_1(s)+e^{t-s}x_2(s)]\,\ud s&=f_1(t),\\
			\int_0^t e^{2t-s}x_1(s)\,\ud s&=f_2(t)
		\end{aligned}\right.
	\end{align*}
	on~$I:=[0,1]$. We choose $f_1(t)$~and $f_2(t)$ such that the exact solution is
	$x_1(t)=t\sin(t)$~and $x_2(t)=\cos t$. Note that this time we have both the
	components~$x_p(t)$ satifying $x_p^{(m)}(0)=0$ ($p=1$, $2$) when~$m$ is odd.
\end{example}

In Tables \ref{tab5}~and \ref{tab6}, we list the maximum absolute errors for
the DG solution. Observe that for the first component~$x_1$, the global
convergence order agrees with Theorem~\ref{thr4}. For the second
component~$x_2$, if $m$ is even, then the global convergence order also agrees
with Theorem~\ref{thr4}, whereas if $m$ is odd, then global
superconvergence is evident, and the corresponding order agrees with
Theorem~\ref{thr6}. Tables \ref{tab7}~and \ref{tab8} list the maximum
absolute errors of the local DG solution at points~$t_n+s_rh$
with~$P'_{m+1}(s_r)=0$ when~$m$ is odd, but with~$P'_{m}(s_r)=0$ when~$m$ is
even. The numerical results show that the condition~$x_1^{(m)}(0)=0$ is
necessary, and the local superconvergence results agree with Theorem~\ref{thr6}.

\begin{table}[!t]
	\caption{The maximum absolute DG errors for $x_1$ component of Example \ref{ex2}}\label{tab5}%
	\begin{tabular*}{\columnwidth}{@{\extracolsep\fill}l>{\ttfamily}l>{\ttfamily}l>{\ttfamily}l>{\ttfamily}l@{\extracolsep\fill}}
		\hline
		$N$  & $m=3$    & $m=4$  & $m=5$  & $m=6$ \\
		\hline
		$2^2$ & 5.64E-04 & 5.58E-05 & 2.90E-07 & 6.51E-07 \\
		$2^3$ & 7.08E-05 & 6.75E-06 & 9.09E-09 & 1.91E-08 \\
		$2^4$ & 8.87E-06 & 8.31E-07 & 2.85E-10 & 5.82E-10 \\
		$2^5$ & 1.11E-06 & 1.03E-07 & 8.91E-12 & 1.80E-11 \\
		Order & 2.998     & 3.012     & 4.999     & 5.015 \\
		\hline
	\end{tabular*}
\end{table}

\begin{table}[!t]
	\caption{The maximum absolute DG errors for $x_2$ component of Example \ref{ex2}}\label{tab6}%
	\begin{tabular*}{\columnwidth}{@{\extracolsep\fill}l>{\ttfamily}l>{\ttfamily}l>{\ttfamily}l>{\ttfamily}l@{\extracolsep\fill}}
		\hline
		$N$  & $m=3$    & $m=4$  & $m=5$  & $m=6$ \\
		\hline
		$2^2$ & 8.82E-03 & 1.90E-02 & 2.60E-04 & 9.23E-05 \\
		$2^3$ & 1.95E-03 & 1.02E-02 & 1.38E-05 & 1.36E-05 \\
		$2^4$ & 4.56E-04 & 5.26E-03 & 7.94E-07 & 1.82E-06 \\
		$2^5$ & 1.10E-04 & 2.67E-03 & 4.76E-08 & 2.34E-07 \\
		Order & 2.052    & 0.978     & 4.060     & 2.959 \\
		\hline
	\end{tabular*}
\end{table}

\begin{table}[!t]
	\caption{The local DG errors at $t_n+s_rh$ for $x_1$ component of Example \ref{ex2}}\label{tab7}%
	\begin{tabular*}{\columnwidth}{@{\extracolsep\fill}l>{\ttfamily}l>{\ttfamily}l>{\ttfamily}l>{\ttfamily}l@{\extracolsep\fill}}
		\hline
		$N$  & $m=3$    & $m=4$  & $m=5$  & $m=6$ \\
		\hline
		$2^2$ & 1.00E-05 & 1.98E-06 & 1.95E-07 & 2.45E-08 \\
		$2^3$ & 6.51E-07 & 1.24E-07 & 3.30E-09 & 3.81E-10 \\
		$2^4$ & 4.10E-08 & 7.74E-09 & 5.38E-11 & 6.00E-12 \\
		$2^5$ & 2.56E-09 & 4.84E-10 & 8.46E-13 & 9.38E-14 \\
		Order & 4.001  & 3.999  & 5.990  & 5.999   \\
		\hline
	\end{tabular*}
\end{table}

\begin{table}[!t]
	\caption{The local DG errors at $t_n+s_rh$ for $x_2$ component of Example \ref{ex2}}\label{tab8}%
	\begin{tabular*}{\columnwidth}{@{\extracolsep\fill}l>{\ttfamily}l>{\ttfamily}l>{\ttfamily}l>{\ttfamily}l@{\extracolsep\fill}}
		\hline
		$N$  & $m=3$    & $m=4$  & $m=5$  & $m=6$ \\
		\hline
		$2^2$ & 3.91E-03 & 3.34E-04 & 5.59E-05 & 1.33E-06 \\
		$2^3$ & 9.12E-04 & 8.74E-05 & 3.17E-06 & 9.64E-08 \\
		$2^4$ & 2.21E-04 & 2.22E-05 & 1.90E-07 & 6.31E-09 \\
		$2^5$ & 5.43E-05 & 5.59E-06 & 1.16E-08 & 4.02E-10 \\
		Order & 2.025     & 1.990  & 4.034 & 3.972  \\
		\hline
	\end{tabular*}
\end{table}

\begin{example}\label{ex3}
	Consider the index-$2$ IAE
	\begin{align*}
		\left\{
		\begin{aligned}
			x_1(t)+\int_0^t[(t-s)x_1(s)+e^{t-s}x_2(s)]\,\ud s&=f_1(t),\\
			\int_0^t e^{2t-s}x_1(s)\,\ud s&=f_2(t)
		\end{aligned}\right.
	\end{align*}
	on~$I:=[0,1]$. We choose $f_1(t)$~and $f_2(t)$ such that the exact solution is
	$x_1(t)=\cos t$~and $x_2(t)=e^{-t}$. Note that this time the components
	satisfy $x_1^{(m)}(0)=0$~and $x_2^{(m)}(0)\ne0$ when~$m$ is odd.
\end{example}

Tables~\ref{tab9}--\ref{tab12} list the numerical results corresponding to
Tables~\ref{tab5}--\ref{tab8} for Example~\ref{ex3}, and we obtain similar
convergence and superconvergence results, which means that
the condition~$x_2^{(m)}(0)=0$ is not necessary. The superconvergence
result is typically decided by the superconvergence of first-kind VIEs.

\begin{table}[!t]
	\caption{The maximum absolute DG errors for $x_1$ component of Example \ref{ex3}}\label{tab9}%
	\begin{tabular*}{\columnwidth}{@{\extracolsep\fill}l>{\ttfamily}l>{\ttfamily}l>{\ttfamily}l>{\ttfamily}l@{\extracolsep\fill}}
		\hline
		$N$  & $m=3$    & $m=4$  & $m=5$  & $m=6$ \\
		\hline
		$2^2$ & 1.53E-04 & 1.56E-05 & 1.63E-06 & 1.16E-07 \\
		$2^3$ & 1.93E-05 & 1.88E-06 & 5.11E-08 & 3.43E-09 \\
		$2^4$ & 2.43E-06 & 2.31E-07 & 1.61E-09 & 1.04E-10 \\
		$2^5$ & 3.05E-07 & 2.87E-08 & 5.07E-11 & 3.18E-12 \\
		Order & 2.994     & 3.009        & 4.989   & 5.031  \\
		\hline
	\end{tabular*}
\end{table}

\begin{table}[!t]
	\caption{The maximum absolute DG errors for $x_2$ component of Example \ref{ex3}}\label{tab10}%
	\begin{tabular*}{\columnwidth}{@{\extracolsep\fill}l>{\ttfamily}l>{\ttfamily}l>{\ttfamily}l>{\ttfamily}l@{\extracolsep\fill}}
		\hline
		$N$  & $m=3$    & $m=4$  & $m=5$  & $m=6$ \\
		\hline
		$2^2$ & 2.34E-03 & 4.92E-03 & 4.16E-05 & 1.32E-04 \\
		$2^3$ & 5.26E-04 & 2.65E-03 & 2.31E-06 & 1.70E-05 \\
		$2^4$ & 1.24E-04 & 1.37E-03 & 1.37E-07 & 2.17E-06 \\
		$2^5$ & 3.02E-05 & 6.97E-04 & 7.91E-09 & 2.75E-07 \\
		Order & 2.038     & 0.975     & 4.114    & 2.980 \\
		\hline
	\end{tabular*}
\end{table}

\begin{table}[!t]
	\caption{The local DG errors at $t_n+s_rh$ for $x_1$ component of Example \ref{ex3}}\label{tab11}%
	\begin{tabular*}{\columnwidth}{@{\extracolsep\fill}l>{\ttfamily}l>{\ttfamily}l>{\ttfamily}l>{\ttfamily}l@{\extracolsep\fill}}
		\hline
		$N$  & $m=3$    & $m=4$  & $m=5$  & $m=6$ \\
		\hline
		$2^2$ & 2.55E-06 & 4.96E-07 & 3.26E-08 & 4.13E-09 \\
		$2^3$ & 1.64E-07 & 3.12E-08 & 5.62E-10 & 6.43E-11 \\
		$2^4$ & 1.03E-08 & 1.95E-09 & 9.02E-12 & 1.01E-12 \\
		$2^5$ & 6.41E-10 & 1.22E-10 & 1.41E-13 & 1.60E-14 \\
		Order & 4.006   & 3.999  & 5.999  & 5.980 \\ 
		\hline
	\end{tabular*}
\end{table}

\begin{table}[!t]
	\caption{The local DG errors at $t_n+s_rh$ for $x_2$ component of Example \ref{ex3}}\label{tab12}%
	\begin{tabular*}{\columnwidth}{@{\extracolsep\fill}l>{\ttfamily}l>{\ttfamily}l>{\ttfamily}l>{\ttfamily}l@{\extracolsep\fill}}
		\hline
		$N$  & $m=3$    & $m=4$  & $m=5$  & $m=6$ \\
		\hline
		$2^2$ & 1.06E-03 & 9.01E-05 & 9.16E-06 & 5.22E-06 \\
		$2^3$ & 2.49E-04 & 2.38E-05 & 5.41E-07 & 3.69E-07 \\
		$2^4$ & 6.05E-05 & 6.08E-06 & 3.31E-08 & 2.43E-08 \\
		$2^5$ & 1.49E-05 & 1.53E-06 & 1.95E-09 & 1.55E-09 \\
		Order & 2.022     & 1.991  & 4.085  & 3.971 \\
		\hline
	\end{tabular*}
\end{table}
%---------------------------------------------------------
\section*{Acknowledgments}
Part of the work of the first author was carried out while she was a vistor at University of New South Wales (October 2023--October 2024). She gratefully acknowledges Professor William McLean for his invitation to visit UNSW, as well as his carefully reading and improving this paper.

\section*{Funding}
The work is supported by the National Natural Science Foundation of China (Grant No.~12171122), the Guangdong Provincial Natural Science Foundation of China (2023A1515010818), and the Shenzhen Science and Technology Program (RCJC20210609103755110). H.~Gao acknowledges the support of the China Scholarship Council program (Project ID: 202306120321).
\bibliographystyle{plain}
\bibliography{main}
\end{document}